\documentclass[11pt]{article}
\usepackage[left=1.0in, right=1.0in, top=0.95in, bottom=1.0in]{geometry}
\usepackage{amsmath,amssymb,amsthm,amsfonts}
\usepackage{graphicx}
\usepackage[authoryear]{natbib}
\usepackage{subcaption}
\usepackage{setspace}
\newcommand\T{\rule{0pt}{2.6ex}}       
\newcommand\B{\rule[-1.2ex]{0pt}{0pt}} 

\numberwithin{equation}{section}

\newtheorem{lemma}{Lemma}
\newtheorem{theorem}{Theorem}
\newtheorem{proposition}{Proposition}

\begin{document}

\title{\textbf{\LARGE A one-dimensional diffusion model for overloaded queues
with customer abandonment}
\thanks{Supported in part by MOE AcRF Grant R-266-000-063-133 and NUS GAI Grant R-716-000-006-133.}}
\author{Shuangchi He
\thanks{Department of Industrial and Systems Engineering, National University of Singapore, \tt{heshuangchi@nus.edu.sg}}}
\date{December 15, 2013}
\maketitle
\begin{abstract}
We use an Ornstein--Uhlenbeck (OU) process to approximate
the queue length process in a $\mbox{GI}/\mbox{GI}/n+\mbox{M}$ queue.
This one-dimensional diffusion model is able to produce accurate
performance estimates in two overloaded regimes: In the first regime,
the number of servers is large and the mean patience time is comparable
to or longer than the mean service time; in the second regime, the number
of servers can be arbitrary but the mean patience time is much longer
than the mean service time. Using the diffusion model, we obtain Gaussian
approximations for the steady-state queue length and the steady-state
virtual waiting time. Numerical experiments demonstrate that the
approximate distributions are satisfactory for queues in these two regimes.

To mathematically justify the diffusion model, we formulate the two
overloaded regimes into an asymptotic framework by considering a sequence of
queues. The mean patience time goes to infinity in both asymptotic
regimes, whereas the number of servers approaches infinity in the first
regime but does not change in the second. The OU process is proved to be
the diffusion limit for the queue length processes in both regimes.
A crucial tool for proving the diffusion limit is a functional
central limit theorem for the superposition of time-scaled renewal processes.
We prove that the superposition of $n$ independent, identically distributed
stationary renewal processes, after being centered and scaled in both space
and time, converges in distribution to a Brownian motion as $n$ goes to
infinity.
\end{abstract}

\section{Introduction \label{sec:Introduction}}

Consider a $\mbox{GI}/\mbox{GI}/n+\mbox{M}$ queue. The customer arrival
process of this system is a renewal process and the service times
are independent, identically distributed (iid) nonnegative random
variables. Customers are served by $n$ identical servers. Upon arrival,
a customer gets into service if an idle server is available; otherwise,
he waits in a buffer with infinite room. Waiting customers are served
on the first-come, first-served basis, and the servers are not allowed
to idle if there are customers waiting. Each customer has a random
patience time. When a customer's waiting time exceeds his patience
time, the customer abandons the system without being served. The patience times
are iid following an exponential distribution, and the sequences of
interarrival, service, and patience times are mutually independent.

We are interested in the performance of this queue when it is overloaded,
i.e., the customer arrival rate is greater than the service capacity.
In this case, not all customers are able to receive service and a
fraction of them must abandon the system. We use a simple one-dimensional
diffusion process to approximate the scaled queue length process.
This diffusion process is an Ornstein--Uhlenbeck (OU) process. The
diffusion model is able to produce accurate performance estimates
when the queue is operated in either of the following two overloaded
regimes. In the first regime, the number of servers is large and the
mean patience time is comparable to or longer than the mean service
time. We call it the \emph{many-server overloaded regime}. In the
second regime, the number of servers can be arbitrary, but the mean
patience time is much longer (i.e., on a higher order) than the mean
service time. This regime is referred to as the \emph{long patience
overloaded regime}. In this paper, most efforts are focused on queues
in the many-server overloaded regime.

Queues with customer abandonment are used to model service systems. A call center with
many service agents is a typical example; see \citet{GansETAL03}
for a comprehensive review. Because the rate of incoming calls
changes over time, a call center may become overloaded during the
peak hours of a day. Waiting on a phone line, a customer may hang
up the phone before being connected to an agent. Although customer
abandonment is present in most call centers, empirical studies suggest
that customers are generally patient when they hold the line. It was
reported in \citet{MandelbaumETAL01} and \citet{MandelbaumZeltyn13}
that in the call center of an Israeli bank, the mean customer patience
time was at least several times longer than the mean service time.
The many-server overloaded
regime is thus relevant to call center operations. As pointed out by
\citet{Whitt06}, in service-oriented call centers, staffing costs
usually dominate the expenses of customer delay and abandonment. The
rational operational regime for these systems is the \emph{efficiency-driven
(ED) regime} that emphasizes server utilization over the quality of
service. In the ED regime, the service capacity is set below the customer
arrival rate by a moderate fraction. Because the lost service demands
of abandoning customers compensate for the excess in the arrival rate
over the service capacity, a many-server queue operated in the ED
regime can still achieve reasonable performance. More specifically,
the mean waiting time is comparable to the mean service time, a moderate
fraction of customers abandon the system, and all servers are almost
always busy. The ED regime is closely related to the many-server overloaded
regime studied in this paper.

For queues in the ED regime, a fluid model proposed by \citet{Whitt06}
is useful in estimating several performance measures, including the
fraction of abandonment, the mean queue length, and the mean virtual
waiting time. In the $\mbox{M}/\mbox{M}/n+\mbox{GI}$ setting, the accuracy of the
fluid model was studied by \citet{BassambooRandhawa10}. They proved
that in the steady state, the accuracy gaps of the fluid approximations
for the mean queue length and the rate of customer abandonment do
not increase with the arrival rate. This implies that fluid approximations
could be particularly accurate when the queue is operated in the ED
regime. Such a deterministic model, however, cannot be used to estimate
any nontrivial probability or distribution. In other words, we cannot
estimate the distribution of queue length or customer waiting time using the fluid
model. The performance targets of a service system may require a tail
probability to be less than a specified value, e.g.,
``80\% of customers wait less than 2
minutes.'' A refined model is thus necessary to obtain
such an estimate. The proposed one-dimensional diffusion
model offers a simple yet accurate refinement for the fluid
model. Although the exponential patience time assumption is
somewhat restrictive, by certain modification the diffusion model may
extend to $\mbox{GI}/\mbox{GI}/n+\mbox{GI}$ queues,
allowing for a general patience time
distribution; see Section~\ref{sec:conclusion} for an illustration.

Customers would wait long if the service is of critical importance. When such a
system gets overloaded, customer waiting times will be prolonged
significantly. The long patience overloaded regime is relevant to this
type of systems. One important example is an organ transplant
waiting list: The transplant candidates on the list form a queue;
when an organ is found, the candidate at the top of the list receives
the organ. These candidates may abandon the waiting list either because
of death, or because their health has deteriorated so that transplantation
is no longer appropriate. As the need for organs usually far exceeds
the supply of donors, a transplant candidate may have to wait for
years before transplantation. Such a system must be operated in the
long patience overloaded regime; see \citet{SuZenios04,SuZenios06}.
\citet{JenningsReed12} studied fluid and diffusion models for the
virtual waiting time process of a single-server queue in this regime.

As an OU process, the diffusion model has a Gaussian
stationary distribution. This fact allows us to approximate the steady-state
queue length and virtual waiting time distributions by Gaussian distributions.
The proposed diffusion model, whether for a many-server queue or for
a queue with one or several servers, depends on the interarrival and
service time distributions only through their first two moments. This
is in sharp contrast to the approximate models for many-server queues
in the literature, where the entire service time distribution is built
into the fluid or diffusion equations; see, e.g., \citet{Whitt06},
\citet{KangRamanan10}, \citet{MandelbaumMomcilovic12}, and \citet{Zhang13}.
With a general service time distribution, these approximate models
are either non-Markovian or deterministic. It is difficult to obtain
the steady-state queue length and virtual waiting time distributions
using these models. When the service time distribution is phase-type,
\citet{DaiHeTezcan10} proved a multi-dimensional diffusion limit
for many-server queues in an overloaded regime. As phase-type distributions
can approximate any positive-valued distribution, this model is still
relevant to queues with a general service time distribution. Using
this multi-dimensional model, \citet{DaiHe13} proposed a finite element
algorithm for computing the steady-state queue length distribution.
Although the algorithm is able to produce accurate performance estimates,
the computational complexity increases exponentially as the dimension
of the diffusion model grows. The curse of dimensionality is a serious
issue when the dimension is not small. In contrast, the diffusion
model proposed in this paper is a one-dimensional process that has
an explicit stationary distribution, thus leading to simple performance
formulas.

The one-dimensional diffusion model is rooted in the limit theorems
presented in Section~\ref{sec:Limits}. The diffusion limits for
queues in the many-server overloaded regime and in the long patience
overloaded regime can be found in Theorems~\ref{theorem:many-server}
and~\ref{theorem:long-patience}, respectively. Although the two limit
processes are identical, it is more challenging to prove the
diffusion limit in the many-server regime. In this regime,
we consider a sequence of queues indexed by the number of servers
$n$, and assume that the mean patience time goes to infinity as $n$
goes large. The queue length processes within this asymptotic framework
are scaled in both space and time, with the number of servers and
the mean patience time being the respective scaling factors. This
\emph{space-time scaling} is essential to obtain a
one-dimensional diffusion limit when the service time distribution
is general. In the asymptotic regimes specified in \citet{DaiHeTezcan10}
and \citet{MandelbaumMomcilovic12}, by contrast, the queue length
processes are scaled only in space and not in time. In this case,
only when the service time distribution is exponential, will the
scaled queue length processes converge to a one-dimensional Markov process.

The technique of scaling in both space and time
has been used in \citet{Whitt03,Whitt04,Gurvich04}, and
\citet{Atar12} for many-server queues with an exponential service time
distribution. It is not surprising that the diffusion limits in these
papers are one-dimensional. Theorem~\ref{theorem:many-server}
in our paper demonstrates that by the means of space-time scaling,
many-server queues with a general service time distribution may also have a
one-dimensional diffusion limit when they are overloaded.
The space-time scaling used in our model is similar to
the scaling used in Theorem~4.1 in \citet{Whitt04}, where a sequence of
$\mbox{M}/\mbox{M}/n/r+\mbox{M}$ queues is studied in an overloaded
regime. This regime allows either the number of servers or the mean patience
time or both of them to go to infinity, all of which lead to an OU
limit process. The latter two cases of this regime
correspond to the long patience overloaded regime and the many-server overloaded
regime, respectively. In this sense, Theorems~\ref{theorem:many-server}
and~\ref{theorem:long-patience} in our paper have extended the OU limit
for overloaded queues to a much more general setting.
A critically loaded regime, known as the nondegenerate slowdown regime,
is studied by \citet{Whitt03,Gurvich04}, and \citet{Atar12} for many-server
queues with an exponential service time distribution. In this regime, the
diffusion limit for the queue length processes, which is also scaled in both
space and time, is either a reflected OU process when the patience time distribution
is exponential, or a reflected Brownian motion when there is no abandonment.

The most important tool for proving the diffusion limit in the many-server
regime is a functional central limit theorem (FCLT) for the superposition
of time-scaled, stationary renewal processes, which is presented in
Theorem~\ref{theorem:FCLT}. The well-known FCLT for renewal processes
states that as the scaling factor goes to infinity, a \emph{time-scaled}
renewal process converges in distribution to a Brownian motion.
\citet{Whitt85} proved
an FCLT for the superposition of renewal processes, which
states that the superposition of $n$ iid stationary renewal processes,
after being scaled in \emph{space}, converges in distribution
to a Gaussian process. In general, this Gaussian process is not a
Brownian motion. Theorem~\ref{theorem:FCLT} in our paper is a supplement to these
results. We prove that the superposition of $n$ iid stationary renewal
processes, after being scaled in both \emph{space} and
\emph{time}, converges in distribution to a Brownian motion again.
This theorem allows us to approximate the scaled service completion
process by a Brownian motion, which is the key to
approximating the scaled queue length process by a one-dimensional
diffusion process. To apply this theorem, we consider a
sequence of perturbed systems that are asymptotically equivalent to the original queues but have
simpler dynamics. We assume that servers in a perturbed system
are always busy so that the service completion process
is the superposition of $n$ renewal processes. Using the
simplified dynamics of perturbed systems, we prove the many-server
diffusion limit by a standard continuous mapping approach.

The remainder of the paper is organized as follows. The diffusion model and
the approximate formulas are introduced in Section~\ref{sec:Diffusion}.
Their underlying limit theorems are presented in Section~\ref{sec:Limits}.
We examine the performance formulas by numerical examples in Section~\ref{sec:Numerical}.
Sections~\ref{sec:Proof-queue} and~\ref{sec:Proof-virtual} are dedicated to
the respective proofs of Theorems~\ref{theorem:many-server}
and~\ref{theorem:many-server-virtual}. Future research topics are discussed in Section~\ref{sec:conclusion}. We leave the proof of
Theorem~\ref{theorem:FCLT} to the appendix.

\subsection*{Notation}

All random variables and processes are defined on a common probability
space $(\Omega,\mathcal{F},\mathbb{P})$. We reserve $\mathbb{E}[\cdot]$
for expectation. The symbols $\mathbb{N}$, $\mathbb{N}_{0}$, $\mathbb{R}$,
and $\mathbb{R}_{+}$ are used to denote the sets of positive integers,
nonnegative integers, real numbers, and nonnegative real numbers,
respectively. The space of functions $f:\mathbb{R}_{+}\rightarrow\mathbb{R}$
that are right-continuous on $[0,\infty)$ and have left limits on
$(0,\infty)$ is denoted by $\mathbb{D}$, which is endowed with the
Skorohod $J_{1}$ topology. Given
an arbitrary function $f\in\mathbb{D}$ and a function $g\in\mathbb{D}$
that is nondecreasing and takes values in $\mathbb{R}_{+}$, $f\circ g$
denotes the composed function in $\mathbb{D}$ with $(f\circ g)(t)=f(g(t))$
for $t\geq0$. For a sequence of random variables (or processes) $\{\xi_{n}:n\in\mathbb{N}\}$
taking values in $\mathbb{R}$ (or $\mathbb{D}$), we write $\xi_{n}\overset{\text{a.s.}}{\rightarrow}\xi$
for the almost sure convergence of $\xi_{n}$ to $\xi$ and write
$\xi_{n}\Rightarrow\xi$ for the convergence of $\xi_{n}$ to $\xi$
in distribution, where $\xi$ is a random variable with values in
$\mathbb{R}$ (or a process with values in $\mathbb{D}$). For a random
variable $\xi$ with mean $m_{\xi}>0$ and variance $\sigma_{\xi}^{2}\geq0$,
the squared coefficient of variation of $\xi$ is defined by $c_{\xi}^{2}=\sigma_{\xi}^{2}/m_{\xi}^{2}$.
For any $a,b\in\mathbb{R}$, $a^{+}=\max\{a,0\}$, $a^{-}=\max\{-a,0\}$,
$a\vee b=\max\{a,b\}$, and $a\wedge b=\min\{a,b\}$. We use $e$
for the identity function on $\mathbb{R}_{+}$ and $\chi$ for the
constant one function on $\mathbb{R}_{+}$, i.e., $e(t)=t$ and $\chi(t)=1$
for $t\geq0$. For a fixed $s\geq0$, we use $e^{s}$ to denote the
identity function on $\mathbb{R}_{+}$ that is capped by $s$, i.e.,
$e^{s}(t)=s\wedge t$ for $t\geq0$.

\section{Diffusion model and performance formulas}
\label{sec:Diffusion}

Let $\lambda$ be the customer arrival rate and $\mu$
be the service rate of each server. Assume that both interarrival
times and service times have finite variances,
with squared coefficients of variations $c_{A}^{2}$ and $c_{S}^{2}$,
respectively. As the queue is overloaded, the traffic intensity satisfies
$\rho=\lambda/(n\mu)>1$. If all servers are almost always busy,
the fraction of abandoning customers can be approximated by
\begin{equation}
\alpha \approx\frac{\rho-1}{\rho}.
\label{eq:abd-fraction}
\end{equation}
Let $\gamma$ be the mean patience time.
Since patient times are exponentially distributed, each waiting
customer abandons the system at rate $1/\gamma$. When the
queue is in the steady state, the total abandonment rate from the
buffer must be around $n\mu(\rho-1)$ by the conservation of flow.
Hence, the mean queue length (i.e., the mean number of customers in the buffer)
can be approximated by
\begin{equation}
q \approx n\mu(\rho-1)\gamma.
\label{eq:mean-queue}
\end{equation}
Let $X(t)$ be the number of customers in the system at time~$t$, which fluctuates around $n+q$
as the queue comes into the steady state.
To describe the evolution of queue length around the mean,
we introduce a scaled version of $X$ by
\[
\tilde{X}(t)=\frac{1}{\sqrt{n\gamma}}(X(\gamma t)-n-q).
\]
We call $\tilde{X}$ the \emph{scaled queue length process}. Note
that after the mean is removed, $X$ is scaled in both space and time.
Besides the commonly used scaling in space by the number of servers,
we also change the time scale of the process with the mean patience
time as the factor. We propose to use a one-dimensional diffusion
process $\hat{X}$ to approximate the scaled queue length process.
The initial value of $\hat{X}$ may be taken as $\hat{X}(0)=\tilde{X}(0)$.
This diffusion process is an OU process that satisfies the following
stochastic differential equation
\begin{equation}
\hat{X}(t)=\hat{M}(t)-\int_{0}^{t}\hat{X}(u)\,\mathrm{d}u\quad\mbox{for }t\geq0.\label{eq:OU}
\end{equation}
Here, $\hat{M}$ is a driftless Brownian motion with variance
$\mu(\rho c_{A}^{2}+c_{S}^{2}+\rho-1)$ and $\hat{M}(0)=\hat{X}(0)$.

The OU process is a reasonable model because the queue length process
is mean-reverting: At any time, the instantaneous customer abandonment
rate from the buffer is proportional to the queue length; when the
queue length is either too long or too short, the increased or decreased
abandonment rate will pull it back to the equilibrium level. For the
diffusion model to be accurate, the mean patience time $\gamma$,
serving as the scaling factor in time, should be relatively long compared
with the mean service time. More specifically, this model is able
to produce satisfactory performance approximations for queues
in the many-server overloaded regime and queues in the long patience overloaded
regime. The diffusion model in these two regimes is formalized by
Theorems~\ref{theorem:many-server} and~\ref{theorem:long-patience},
where both regimes are built into an asymptotic framework
and $\hat{X}$ is proved to be the limit of the
scaled queue length processes.
Although the mean patience time goes to infinity in the asymptotic
framework, the diffusion model may still work well
for a many-server queue when the mean patience time is comparable
to or just several times longer than the mean service time. If
the number of servers is not many, however, the mean patience time
is usually required to be much longer than the mean service time.
See Section~\ref{sec:Numerical} for further discussion.

The one-dimensional diffusion model yields useful performance approximations.
It is well known that the stationary distribution of the OU process
is Gaussian. In particular, $\hat{X}$ has a Gaussian stationary distribution with
mean $0$ and variance $\mu(\rho c_{A}^{2}+c_{S}^{2}+\rho-1)/2$. Let
$X(\infty)$ be the stationary number of
customers in the system and $\tilde{X}(\infty)$ be the scaled version.
Because $\hat{X}$ is an approximation of $\tilde{X}$, their steady-state
distributions are expected to be close, i.e.,
\begin{equation}
\mathbb{P}[\tilde{X}(\infty)>a]\approx1-\Phi\bigg(\frac{\sqrt{2}a}{\sqrt{\mu(\rho c_{A}^{2}+c_{S}^{2}+\rho-1)}}\bigg)\quad\mbox{for }a\in\mathbb{R},
\label{eq:steadyQ}
\end{equation}
where $\Phi$ is the standard Gaussian distribution function.
As a result, the steady-state queue length approximately
follows a Gaussian distribution with mean $q$ and variance
\begin{equation}
\sigma_{Q}^2\approx \frac{n\gamma\mu}{2}(\rho c_{A}^{2}+c_{S}^{2}+\rho-1).
\label{eq:var-queue}
\end{equation}

Suppose that at time $s\geq0$, a hypothetical customer with infinite
patience arrives at the queue. Let $W(s)$ be the amount of time this
hypothetical customer has to wait before getting into service.
This waiting time is called the \emph{virtual waiting time} at~$s$.
In the steady state, the virtual waiting time process fluctuates around
its mean $w$, which can be determined as follows. As
the patience time distribution is exponential with mean $\gamma$,
the fraction of customers whose patience times are longer than $w$
is $\exp(-w/\gamma)$. This fraction should be approximately equal to the fraction
of customers who eventually receive service, so that $\exp(-w/\gamma)\approx1/\rho$,
or
\begin{equation}
w \approx \gamma\log\rho.
\label{eq:mean-virtual}
\end{equation}
We are interested in the distribution of the steady-state virtual waiting
time. Let $W(\infty)$ be the virtual waiting
time in the steady state, which has a scaled version
\[
\tilde{W}(\infty)=\sqrt{n\gamma^{-1}}(W(\infty)-w).
\]
Theorems~\ref{theorem:many-server-virtual} and~\ref{theorem:long-patience-virtual}
in Section~\ref{sec:Limits} imply that $\tilde{W}(\infty)$ approximately
follows a Gaussian distribution with mean $0$ and variance
$(c_{A}^{2}+\rho c_{S}^{2}+\rho-1)/(2\mu\rho)$, i.e.,
\begin{equation}
\mathbb{P}[\tilde{W}(\infty)>a]\approx 1-
\Phi\bigg(\frac{a\sqrt{2\mu\rho}}{\sqrt{c_{A}^{2}+\rho c_{S}^{2}+\rho-1}}\bigg)\quad\mbox{for }a\in\mathbb{R}.
\label{eq:steadyW}
\end{equation}
Hence, the virtual waiting time in the steady state approximately
follows a Gaussian distribution with mean $w$ and
variance
\begin{equation}
\sigma_{W}^{2}\approx \frac{\gamma}{2n\mu\rho}(c_{A}^{2}+\rho c_{S}^{2}+\rho-1).
\label{eq:var-virtual}
\end{equation}
Formulas \eqref{eq:steadyQ} and \eqref{eq:steadyW} provide approximate
distributions for the queue length and virtual waiting time in the steady state.
They will be examined in Section~\ref{sec:Numerical}.

\section{Limit theorems \label{sec:Limits}}

In this section, we state the underlying limit theorems for the
diffusion model and the approximate formulas.
The theorems for queues in the many-server overloaded regime and in
the long patience overloaded regime are presented in
Sections~\ref{subsec:many-server-limit}
and~\ref{subsec:long-patience-limit}, respectively.

\subsection{Limits in the many-server overloaded regime
\label{subsec:many-server-limit}}

To formulate the many-server overloaded regime, let us consider a
sequence of $\mbox{G}/\mbox{GI}/n+\mbox{M}$ queues indexed by
the number of servers $n$. The arrival processes in these queues
are not required to be renewal. In each queue, the number of initial
customers, the arrival process, the sequence of service times, and
the sequence of patience times are mutually independent. All these
queues have the same traffic intensity $\rho>1$ and the same service
time distribution. Because the service rate $\mu$ is invariant, the
arrival rate of the $n$th system is
\begin{equation}
\lambda_{n}=n\rho\mu.\label{eq:lambda_n}
\end{equation}
We assume that the mean patience time goes to infinity as $n$ goes
large, i.e.,
\begin{equation}
\gamma_{n}\rightarrow\infty\quad\mbox{as }n\rightarrow\infty.\label{eq:gamma_n}
\end{equation}

Let $F$ be the distribution function of service times. As in \citet{Whitt85},
a mild regularity condition is imposed on $F$, i.e.,
\begin{equation}
\limsup_{t\downarrow0}t^{-1}(F(t)-F(0))<\infty.\label{eq:conditionF}
\end{equation}
We also assume that $F$ has a finite third moment, i.e.,
\begin{equation}
\int_{0}^{\infty}t^{3}\,\mathrm{d}F(t)<\infty.
\label{eq:m3}
\end{equation}
Then, the equilibrium distribution of $F$ is given by
\[
F_{e}(t)=\mu\int_{0}^{t}(1-F(u))\,\mathrm{d}u\quad\mbox{for }t\geq0\mbox{.}
\]
We assign service times to customers according to the following procedure. Let $\{\xi_{j,k}:j,k\in\mathbb{N}\}$
be a double sequence of independent nonnegative random variables.
For each $j\in\mathbb{N}$, we assume that
\begin{equation}
\xi_{j,1}\mbox{ follows distribution }F_{e}\mbox{ and }\xi_{j,k}\mbox{ follows distribution }F\mbox{ for }k\geq2.\label{eq:service}
\end{equation}
In the $n$th system, assume that all $n$ servers are busy at time~$0$.
For $j=1,\ldots,n$, $\xi_{j,1}$ is assigned to the initial customer
served by the $j$th server as the residual service time at time
$0$. For $k\geq2$, $\xi_{j,k}$ is the service
time of the $k$th customer served by the $j$th server. By this assignment,
for all $j,k\in\mathbb{N}$, the $k$th service time by the $j$th
server is identical in all systems that have at least $j$ servers.

Let $E_{n}(t)$ be the number of arrivals in the $n$th system during
time interval $(0,t]$. Define the diffusion-scaled arrival process
$\tilde{E}_{n}$ by
\[
\tilde{E}_{n}(t)=\frac{1}{\sqrt{n\gamma_{n}}}(E_{n}(\gamma_{n}t)-\lambda_{n}\gamma_{n}t).
\]
Let $N$ be a renewal process whose interrenewal times have mean $1$
and variance $c_{A}^{2}$. If $E_{n}$ is renewal with
$E_{n}(t)=N(\lambda_{n}t)$, it follows from \eqref{eq:lambda_n}
and the FCLT for renewal processes that
\begin{equation}
\tilde{E}_{n}\Rightarrow\hat{E}\quad\mbox{as }n\rightarrow\infty\mbox{,}\label{eq:E}
\end{equation}
where $\hat{E}$ is a driftless Brownian motion with variance $\rho\mu c_{A}^{2}$
and $\hat{E}(0)=0$. To allow for more general arrival
processes, we take \eqref{eq:E} as an assumption rather than require
each $E_{n}$ to be renewal. Let $X_{n}(t)$ be the number of customers
in the $n$th system at time $t$, which has a diffusion-scaled version
\[
\tilde{X}_{n}(t)=\frac{1}{\sqrt{n\gamma_{n}}}(X_{n}(\gamma_{n}t)-n-n\mu(\rho-1)\gamma_{n}).
\]
We assume that there exists a random variable $\hat{X}(0)$ such that
\begin{equation}
\tilde{X}_{n}(0)\Rightarrow\hat{X}(0)\quad\mbox{as }n\rightarrow\infty\mbox{.}\label{eq:initial}
\end{equation}

The first theorem states the diffusion limit for queue length processes
in the many-server overloaded regime. It justifies the diffusion model
when the queue has many servers.

\begin{theorem}
\label{theorem:many-server}
Let $\hat{X}$ be the OU process given by \eqref{eq:OU}. Assume that
the sequence of $\mbox{G}/\mbox{GI}/n+\mbox{M}$ queues, each indexed
by the number of servers $n$, satisfies \eqref{eq:lambda_n}--\eqref{eq:initial}
with $\rho>1$.
Then,
\[
\tilde{X}_{n}\Rightarrow\hat{X}\quad\mbox{as }n\rightarrow\infty\mbox{.}
\]
\end{theorem}

The second theorem concerns virtual waiting times in these queues.
Let $W_{n}(s)$ be the virtual waiting time at $s\geq0$ in the
$n$th queue. A scaled version is defined by
\[
\bar{W}_{n}(s)=\gamma_{n}^{-1}W_{n}(\gamma_{n}s).
\]
By \eqref{eq:mean-virtual}, we expect $\bar{W}_{n}(s)$ to be close to $\log\rho$. To obtain a refined approximation, we further
define
\[
\tilde{W}_{n}(s)=\sqrt{n\gamma_{n}}(\bar{W}_{n}(s)-\log\rho),
\]
which describes the variation of the virtual waiting time
around the mean. Theorem~\ref{theorem:many-server-virtual}
states that $\tilde{W}_{n}(s)$ converges in distribution as $n$ goes
large.

Let us introduce several processes to state this theorem.
Fix $s\geq0$. Let $\hat{G}^{s}$ be a standard Brownian motion (the
superscript emphasizes that the process may change with~$s$) and $\hat{B}$
be a driftless Brownian motion with variance $\mu c_{S}^{2}$ and
$\hat{B}(0)=0$. Assume that $\hat{X}(0)$, $\hat{E}$,
$\hat{G}^{s}$, and $\hat{B}$ are mutually independent. Define a function
$y^{s}:\mathbb{R}_{+}\rightarrow\mathbb{R}$ by
\begin{equation}
y^{s}(t)=\begin{cases}
(\rho-1)\mu & \mbox{for }0\leq t<s,\\
(\rho\exp(s-t)-1)\mu & \mbox{for }s\leq t<s+\log\rho,\\
-\mu(t-s-\log\rho) & \mbox{for }t\geq s+\log\rho.
\end{cases}\label{eq:ys}
\end{equation}

\begin{theorem}
\label{theorem:many-server-virtual}
Under the conditions of Theorem~\ref{theorem:many-server}, for any
given $s\geq0$,
\[
\tilde{W}_{n}(s)\Rightarrow\mu^{-1}\hat{Y}^{s}(s+\log\rho)\quad\mbox{as }n\rightarrow\infty,
\]
where $\hat{Y}^{s}$ satisfies the stochastic differential
equation
\[
\hat{Y}^{s}(t)=\hat{X}(0)+\hat{E}(s\wedge t)-\hat{B}(t)-\hat{G}^{s}\Big(\int_{0}^{t}y^{s}(u)\,\mathrm{d}u\Big)-\int_{0}^{t}\hat{Y}^{s}(u)\,\mathrm{d}u
\quad\mbox{for } 0\leq t\leq s+\log\rho.
\]
In particular,
\begin{multline}
\hat{Y}^{s}(s+\log\rho)=\exp(-s-\log\rho)\Big(\hat{X}(0)+\int_{0}^{s}\exp(u)\,\mathrm{d}\hat{E}(u)-\int_{0}^{s+\log\rho}\exp(u)\,\mathrm{d}\hat{B}(u)\\
-\int_{0}^{s+\log\rho}y^{s}(u)^{1/2}\exp(u)\,\mathrm{d}\hat{G}^{s}(u)\Big).
\label{eq:Ystop}
\end{multline}
\end{theorem}

Put $\hat{W}(s)=\hat{Y}^{s}(s+\log\rho)/\mu$. As $s$ goes large, $\hat{W}(s)$
converges in distribution to a Gaussian random variable with
mean $0$ and variance $(c_{A}^{2}+\rho c_{S}^{2}+\rho-1)/(2\mu\rho)$,
which leads to formula \eqref{eq:steadyW}.

The third theorem plays an essential role in proving Theorems~\ref{theorem:many-server}
and \ref{theorem:many-server-virtual}. It is an FCLT for the superposition
of time-scaled, stationary renewal processes. These renewal processes
are defined as follows. For $t\geq0$ and $j\in\mathbb{N}$, let
\begin{equation}
N_{j}(t)=\max\{k\in\mathbb{N}_{0}:\xi_{j,1}+\cdots+\xi_{j,k}\leq t\}.\label{eq:renewal}
\end{equation}
 As a convention, we take $N_{j}(t)=0$ if $\xi_{j,1}>t$. By \eqref{eq:service}, each $N_{j}$ is a delayed renewal process
with delay distribution $F_{e}$ and interrenewal distribution $F$.
Because $F_{e}$ is the equilibrium distribution of $F$, $\{N_{j}:j\in\mathbb{N}\}$
is a sequence of iid stationary renewal processes.

\begin{theorem}
\label{theorem:FCLT} Let $\{N_{j}:j\in\mathbb{N}\}$ be a sequence
of iid stationary renewal processes, i.e., the delay distribution
$F_{e}$ of each renewal process is the equilibrium distribution of the
interrenewal distribution $F$. Assume that $F$ has mean $1/\mu$ and
satisfies \eqref{eq:conditionF} and \eqref{eq:m3}. Let
\begin{equation}
B_{n}(t)=\sum_{j=1}^{n}N_{j}(t)\label{eq:B}
\end{equation}
and $\{\gamma_{n}:n\in\mathbb{N}\}$ be a sequence of positive numbers
such that $\gamma_{n}\rightarrow\infty$ as $n\rightarrow\infty$.
Then,
\[
\tilde{B}_{n}\Rightarrow\hat{B}\quad\mbox{as }n\rightarrow\infty\mbox{,}
\]
where
\begin{equation}
\tilde{B}_{n}(t)=\frac{1}{\sqrt{n\gamma_{n}}}(B_{n}(\gamma_{n}t)-n\mu\gamma_{n}t)\label{eq:B_tilde}
\end{equation}
and $\hat{B}$ is a driftless Brownian motion with variance $\mu c_{S}^{2}$
and $\hat{B}(0)=0$.
\end{theorem}

Let us compare Theorem~\ref{theorem:FCLT}
with two other FCLTs. Consider the sequence
of iid stationary renewal processes $\{N_{j}:j\in\mathbb{N}\}$.
By the FCLT for renewal processes, $\{(N_{1}(\ell t)-\ell\mu t)/\sqrt{\ell}:t\geq 0\}$ converges in
distribution to a Brownian motion as $\ell$ goes to infinity; see
Theorem~5.11 in \citet{ChenYao01}. Clearly, the increments of
this time-scaled renewal process become independent of
its history as the scaling factor gets large. \citet{Whitt85}
proved an FCLT for the superposition of stationary renewal processes.
It states that $\{\sum_{j=1}^{n}(N_{j}(t)-\mu t)/\sqrt{n}:t\geq 0\}$
converges in distribution to a zero-mean Gaussian process that has stationary
increments and continuous paths. In this FCLT, the superposition process is
scaled in space only. The covariance function of each stationary renewal
process is retained in the limit Gaussian process, which, in
general, is not a Brownian motion; see Theorem~2 in \citet{Whitt85}.
In our theorem, each superposition process is scaled in both space
and time. Squeezing the time scale erases the dependence of the
increments of $\tilde{B}_{n}$ to its history. The limit of these
space-time scaled superposition processes is thus
a Gaussian process with independent, stationary
increments and continuous paths, which must be a Brownian motion.

In the many-server overloaded regime, all servers of a queue
are nearly always busy. The service completion process is thus
almost identical to a superposition of many renewal processes.
Theorem~\ref{theorem:FCLT} implies that it is possible to approximate
the scaled service completion process by a Brownian motion.
This approximation enables us to explore a simple one-dimensional
diffusion model, which is able to capture the dynamics of a many-server
queue with a general service time distribution, by zooming out our
view in both space and time.

\subsection{Limits in the long patience overloaded regime
\label{subsec:long-patience-limit}}

To formulate the long patience overloaded regime, we fix the number
of servers $n$ and consider a sequence of $\mbox{G}/\mbox{GI}/n+\mbox{M}$
queues indexed by $k\in\mathbb{N}$. All these queues share the
same arrival process, the same service distribution, and thus the same
traffic intensity $\rho>1$. We assume that the mean patience time
in the $k$th queue goes to infinity as $k$ goes large, i.e.,
\begin{equation}
\gamma_{k}\rightarrow \infty\quad\mbox{as }k\rightarrow\infty.
\label{eq:gamma_k}
\end{equation}

Let $E$ be the common arrival process of these queues, which has a
diffusion-scaled version
\[
\tilde{E}_{k}(t) = \frac{1}{\sqrt{n\gamma_{k}}}(E(\gamma_{k}t)-\lambda\gamma_{k}t).
\]
Assume that
\begin{equation}
\tilde{E}_{k}\Rightarrow \hat{E}\quad\mbox{as }k\rightarrow\infty,
\label{eq:Ek}
\end{equation}
where $\hat{E}$ is a drift less Brownian motion with variance $\rho\mu c_{A}^{2}$
and $\hat{E}(0)=0$. Let $X_{k}(t)$ be the number of
customers in the $k$th system at time $t$. Put
\[
\tilde{X}_{k}(t) = \frac{1}{\sqrt{n\gamma_{k}}}(X_{k}(\gamma_{k}t)-n-n\mu(\rho-1)\gamma_{k}).
\]
We assume that there exists a random variable $\hat{X}(0)$ such that
\begin{equation}
\tilde{X}_{k}\Rightarrow \hat{X}(0)\quad\mbox{as }k\rightarrow\infty.
\label{eq:initial_k}
\end{equation}

\begin{theorem}
\label{theorem:long-patience}
Let $\hat{X}$ be the OU process given by \eqref{eq:OU} and $n$ be a fixed
positive integer. Assume that the sequence of $\mbox{G}/\mbox{GI}/n+\mbox{M}$
queues, indexed by $k\in\mathbb{N}$, satisfies \eqref{eq:gamma_k}--\eqref{eq:initial_k} with $\rho>1$.
Then,
\[
\tilde{X}_{k}\Rightarrow\hat{X}\quad\mbox{as }k\rightarrow\infty\mbox{.}
\]
\end{theorem}

In the $k$th queue, let $W_{k}(s)$ be the virtual waiting time at $s\geq 0$,
which has a scaled version
\[
\tilde{W}_{k}(s) = \sqrt{n\gamma_{k}}(\gamma_{k}^{-1}W_{k}(\gamma_{k}s)-\log\rho).
\]
\begin{theorem}
\label{theorem:long-patience-virtual}
Under the conditions of Theorem~\ref{theorem:long-patience}, for any
given $s\geq 0$,
\[
\tilde{W}_{k}(s)\Rightarrow\mu^{-1}\hat{Y}^{s}(s+\log\rho)\quad\mbox{as }k\rightarrow\infty,
\]
where $\hat{Y}^{s}$ is the diffusion process defined in
Theorem~\ref{theorem:many-server-virtual}.
\end{theorem}

Because all queues have the same number of servers, we need the FCLT
for renewal processes, instead of
Theorem~\ref{theorem:FCLT}, in proving Theorems~\ref{theorem:long-patience}
and~\ref{theorem:long-patience-virtual}. With minor modification, one can follow the proofs of
Theorems~\ref{theorem:many-server} and~\ref{theorem:many-server-virtual}
to finish these proofs. We would not include them
in this paper.

\begin{table}[t]
\caption{Performance estimates for an $\mbox{M}/\mbox{GI}/100+\mbox{M}$ queue
with $\mu=1.0$ and $\rho=1.2$; simulation results (with $95\%$ confidence intervals)
are compared with approximate results (in \emph{italics}).}
\centering %
\footnotesize{
\begin{tabular}{l|lllll}
\hline
\multicolumn{1}{c|}{} & \multicolumn{1}{c}{} & \multicolumn{2}{c}{Queue length} & \multicolumn{2}{c}{Virtual waiting time}\T\tabularnewline
\multicolumn{1}{l|}{Patience} & \multicolumn{1}{l}{Abd. fraction} & \multicolumn{1}{l}{Mean} & \multicolumn{1}{l}{Variance} & \multicolumn{1}{l}{Mean} & \multicolumn{1}{l}{Variance}\B\tabularnewline
\hline
\multicolumn{1}{c|}{} & \multicolumn{5}{c}{$\mbox{M}/\mbox{D}/100+\mbox{M}$}\T\B\tabularnewline
\cline{2-6}
$\gamma=1.0$  & $0.1668$  & $20.02$  & $73.11$  & $0.1851$  & $0.005322$\T\tabularnewline
 & $\pm0.000020$  & $\pm0.0034$  & $\pm0.038$  & $\pm0.000028$  & $\pm0.0000030$\tabularnewline
 & $\emph{0.1667}$  & $\emph{20.00}$  & $\emph{70.00}$  & $\emph{0.1823}$  & $\emph{0.005000}$\B\tabularnewline
$\gamma=5.0$  & $0.1667$  & $99.99$  & $364.1$  & $0.9142$  & $0.02639$ \T\tabularnewline
 & $\pm0.000021$  & $\pm0.017$  & $\pm4.3$  & $\pm0.00014$  & $\pm0.00042$\tabularnewline
 & $\emph{0.1667}$  & $\emph{100.0}$  & $\emph{350.0}$  & $\emph{0.9116}$  & $\emph{0.02500}$\B\tabularnewline
$\gamma=10$  & $0.1667$  & $200.0$  & $749.2$  & $1.826$  & $0.05487$ \T\tabularnewline
 & $\pm0.000021$  & $\pm0.035$  & $\pm1.2$  & $\pm0.00030$  & $\pm0.000086$\tabularnewline
 & $\emph{0.1667}$  & $\emph{200.0}$  & $\emph{700.0}$  & $\emph{1.823}$  & $\emph{0.05000}$\B\tabularnewline
\hline
\multicolumn{1}{c|}{} & \multicolumn{5}{c}{$\mbox{M}/\mbox{E}_{2}/100+\mbox{M}$}\T\B\tabularnewline
\cline{2-6}
$\gamma=1.0$  & $0.1672$  & $20.07$  & $97.08$  & $0.1869$  & $0.007799$ \T\tabularnewline
 & $\pm0.000040$  & $\pm0.0062$  & $\pm0.041$  & $\pm0.000055$  & $\pm0.0000033$\tabularnewline
 & $\emph{0.1667}$  & $\emph{20.00}$  & $\emph{95.00}$  & $\emph{0.1823}$  & $\emph{0.007500}$\B\tabularnewline
$\gamma=5.0$  & $0.1666$  & $99.97$  & $481.0$  & $0.9152$  & $0.03812$ \T\tabularnewline
 & $\pm0.000043$  & $\pm0.035$  & $\pm0.63$  & $\pm0.00031$  & $\pm0.000049$\tabularnewline
 & $\emph{0.1667}$  & $\emph{100.0}$  & $\emph{475.0}$  & $\emph{0.9116}$  & $\emph{0.03750}$\B\tabularnewline
$\gamma=10$  & $0.1666$  & $199.9$  & $956.4$  & $1.827$  & $0.07567$ \T\tabularnewline
 & $\pm0.000042$  & $\pm0.066$  & $\pm2.0$  & $\pm0.00058$  & $\pm0.00015$\tabularnewline
 & $\emph{0.1667}$  & $\emph{200.0}$  & $\emph{950.0}$  & $\emph{1.823}$  & $\emph{0.07500}$\B\tabularnewline
\hline
\multicolumn{1}{c|}{} & \multicolumn{5}{c}{$\mbox{M}/\mbox{LN}/100+\mbox{M}$}\T\B\tabularnewline
\cline{2-6}
$\gamma=1.0$  & $0.1677$  & $20.12$  & $115.4$  & $0.1884$  & $0.009753$ \T\tabularnewline
 & $\pm0.000038$  & $\pm0.0053$  & $\pm0.050$  & $\pm0.000050$  & $\pm0.0000043$\tabularnewline
 & $\emph{0.1667}$  & $\emph{20.00}$  & $\emph{146.0}$  & $\emph{0.1823}$  & $\emph{0.01260}$\B\tabularnewline
$\gamma=5.0$  & $0.1666$  & $99.97$  & $670.7$  & $0.9171$  & $0.05719$ \T\tabularnewline
 & $\pm0.000038$  & $\pm0.026$  & $\pm0.70$  & $\pm0.00025$  & $\pm0.000063$\tabularnewline
 & $\emph{0.1667}$  & $\emph{100.0}$  & $\emph{730.0}$  & $\emph{0.9116}$  & $\emph{0.06300}$\tabularnewline
$\gamma=10$  & $0.1666$  & $199.9$  & $1385$  & $1.829$  & $0.1187$ \T\tabularnewline
 & $\pm0.000039$  & $\pm0.053$  & $\pm1.8$  & $\pm0.00050$  & $\pm0.00016$\tabularnewline
 & $\emph{0.1667}$  & $\emph{200.0}$  & $\emph{1460}$  & $\emph{1.823}$  & $\emph{0.1260}$\B\tabularnewline
\hline
\end{tabular}
}
\label{table:measure100}
\end{table}

\begin{table}[t]
\caption{Tail probabilities for queue length and virtual waiting time in an $\mbox{M}/\mbox{GI}/100+\mbox{M}$ queue
with $\mu=1.0$ and $\rho=1.2$; simulation results (with $95\%$ confidence intervals)
are compared with diffusion approximations (in \emph{italics}).}
\centering %
\footnotesize{
\begin{tabular}{l|llllll}
\hline
 & \multicolumn{3}{c}{$\mathbb{P}[\tilde{X}(\infty)>a]$} & \multicolumn{3}{c}{$\mathbb{P}[\tilde{W}(\infty)>a]$}\T\tabularnewline
\multicolumn{1}{c|}{Patience} & \multicolumn{1}{c}{$a=0.5$} & \multicolumn{1}{c}{$a=1.0$} & \multicolumn{1}{c}{$a=2.0$} & \multicolumn{1}{c}{$a=0.5$} & \multicolumn{1}{c}{$a=1.0$} & \multicolumn{1}{c}{$a=2.0$}\B\tabularnewline
\hline
\multicolumn{1}{c|}{} & \multicolumn{6}{c}{$\mbox{M}/\mbox{D}/100+\mbox{M}$}\T\B\tabularnewline
\cline{2-7}
$\gamma=1.0$  & $0.2559$  & $0.1131$  & $0.01140$  & $0.2584$  & $0.09269$  & $0.003869$\T\tabularnewline
 & $\pm0.00014$  & $\pm0.000089$  & $\pm0.000031$  & $\pm0.00014$  & $\pm0.000078$  & $\pm0.000018$\B\tabularnewline
$\gamma=5.0$  & $0.2707$  & $0.1200$ & $0.01120$  & $0.2505$  & $0.08689$  & $0.003138$\T\tabularnewline
 & $\pm0.0013$  & $\pm0.0013$  & $\pm0.00031$  & $\pm0.0019$  & $\pm0.0016$  & $\pm0.00017$\B\tabularnewline
$\gamma=10$  & $0.2840$  & $0.1252$ & $0.01089$  & $0.2539$  & $0.09004$  & $0.003419$\T\tabularnewline
 & $\pm0.00049$  & $\pm0.00029$  & $\pm0.000093$  & $\pm0.00046$  & $\pm0.00023$  & $\pm0.000050$\B\tabularnewline
 & $\emph{0.2750}$  & $\emph{0.1160}$  & $\emph{0.008414}$  & $\emph{0.2398}$  & $\emph{0.07865}$  & $\emph{0.002339}$\T\B\tabularnewline
\hline
\multicolumn{1}{c|}{} & \multicolumn{6}{c}{$\mbox{M}/\mbox{E}_{2}/100+\mbox{M}$}\T\B\tabularnewline
\cline{2-7}
$\gamma=1.0$  & $0.2865$  & $0.1472$  & $0.02314$  & $0.3007$  & $0.1422$  & $0.01596$\T\tabularnewline
 & $\pm0.00023$  & $\pm0.00015$  & $\pm0.000044$  & $\pm0.00023$  & $\pm0.00015$  & $\pm0.000039$\B\tabularnewline
$\gamma=5.0$  & $0.2972$  & $0.1523$ & $0.02261$  & $0.2884$  & $0.1302$  & $0.01215$\T\tabularnewline
 & $\pm0.00056$  & $\pm0.00041$  & $\pm0.00014$  & $\pm0.00055$  & $\pm0.00039$  & $\pm0.000098$\B\tabularnewline
$\gamma=10$  & $0.3057$  & $0.1538$ & $0.02095$  & $0.2859$  & $0.1279$  & $0.01151$\T\tabularnewline
 & $\pm0.00074$  & $\pm0.00059$  & $\pm0.00021$  & $\pm0.00073$  & $\pm0.00054$  & $\pm0.00014$\B\tabularnewline
 & $\emph{0.3040}$  & $\emph{0.1525}$  & $\emph{0.02009}$  & $\emph{0.2819}$  & $\emph{0.1241}$  & $\emph{0.01046}$\T\B\tabularnewline
\hline
\multicolumn{1}{c|}{} & \multicolumn{6}{c}{$\mbox{M}/\mbox{LN}/100+\mbox{M}$}\T\B\tabularnewline
\cline{2-7}
$\gamma=1.0$  & $0.3041$  & $0.1697$  & $0.03416$  & $0.3221$  & $0.1726$  & $0.03005$\T\tabularnewline
 & $\pm0.00018$  & $\pm0.00012$  & $\pm0.000046$  & $\pm0.00019$  & $\pm0.00013$  & $\pm0.000043$\B\tabularnewline
$\gamma=5.0$  & $0.3259$  & $0.1917$ & $0.04421$  & $0.3247$  & $0.1802$  & $0.03442$\T\tabularnewline
 & $\pm0.00036$  & $\pm0.00026$  & $\pm0.00015$  & $\pm0.00036$  & $\pm0.00026$  & $\pm0.00013$\B\tabularnewline
$\gamma=10$  & $0.3362$  & $0.1978$ & $0.04491$  & $0.3260$  & $0.1829$  & $0.03605$\T\tabularnewline
 & $\pm0.00047$  & $\pm0.00040$  & $\pm0.00021$  & $\pm0.00047$  & $\pm0.00035$  & $\pm0.00020$\B\tabularnewline
 & $\emph{0.3395}$  & $\emph{0.2039}$  & $\emph{0.04894}$  & $\emph{0.3280}$  & $\emph{0.1865}$  & $\emph{0.03740}$\T\B\tabularnewline
\hline
\end{tabular}
}
\label{table:tail100}
\end{table}

\section{Numerical examples
\label{sec:Numerical}}

In this section, we examine the approximate formulas obtained from the
diffusion model by simulation. We assume a Poisson arrival process and an exponential
patience time distribution. All numerical examples have the same
traffic intensity $\rho=1.2$. Different service time distributions, all with
mean $1/\mu=1.0$, are tested in the many-server and long patience
overloaded regimes.

In the simulation examples, the service time distribution may be
deterministic, Erlang (with two stages), or log-normal. These three
distributions are denoted by $\mbox{D}$, $\mbox{E}_2$, and $\mbox{LN}$,
respectively. With $c_{S}^{2}=0$ and $0.5$, respectively, the deterministic and
Erlang distributions are used to represent scenarios where service times
have small to moderate variability. It was reported in \citet{BrownETAL05}
that a log-normal distribution provides a good fit for the service time
data from the call center of an Israeli bank. We also test such a
distribution that yields more variable service times. The log-normal
distribution has $c_{S}^{2}=1.52$, which is identical to the value from the data in \citet{BrownETAL05}.
All simulation results are obtained by averaging $30$ independent runs
and in each run, the queue is simulated for $1.0\times 10^{6}$ time units.

\subsection{Examples in the many-server overloaded regime
\label{subsection:many-server}}

Consider an $\mbox{M}/\mbox{GI}/100+\mbox{M}$ queue. With $n=100$, the
customer arrival rate is $\lambda = n\rho\mu = 120$. We evaluate the
performance of this queue with mean patience time $\gamma = 1.0$,
$5.0$, and $10$, respectively.

The estimates of several performance measures, including the abandonment
fraction, the mean and variance of the steady-state queue length, and the
mean and variance of the steady-state virtual waiting time, are listed in
Table~\ref{table:measure100}. We use \eqref{eq:abd-fraction},
\eqref{eq:mean-queue}, \eqref{eq:var-queue}, \eqref{eq:mean-virtual}, and
\eqref{eq:var-virtual} to obtain the approximate results. Formulas
\eqref{eq:abd-fraction}, \eqref{eq:mean-queue}, and \eqref{eq:mean-virtual}
can be obtained from the fluid model proposed by \citet{Whitt06}. This
fluid model, however, cannot be used to estimate variances.

In Table~\ref{table:measure100}, the approximate results of the
abandonment fraction, the mean queue length, and the mean virtual
waiting time agree with the simulation results very well. This is
consistent with the conclusion drawn by \citet{Whitt06}: The fluid
model is able to produce accurate approximations for mean
performance measures in an overloaded queue with many servers.
As the scaling factor in time, the mean patience time has
an influence on the accuracy of the diffusion model.
Theorems~\ref{theorem:many-server} and~\ref{theorem:many-server-virtual}
imply that diffusion approximations become more accurate as the
mean patience time increases. Comparing the variance results in the
table, however, we can tell that an adequate diffusion approximation
may not require a long mean patience time: With a mean patience
time that is comparable to the mean service time, the approximate variances
are satisfactory when the service
times are deterministic or follow an Erlang distribution.
We may explain this observation as follows.
Because the service completion process is close to a superposition
of renewal processes, a Brownian motion is used implicitly in the
diffusion model to approximate its fluctuation (see Section~\ref{sec:Proof-queue}
for more details). This replacement is supported by Theorem~\ref{theorem:FCLT}.
As we discussed in Section~\ref{subsec:many-server-limit}, by squeezing
the time scale, the increments of
the service completion process become less dependent to the history,
so that a Brownian motion can approximate a space-time scaled version of this process. If the variability of service
times is not large, a moderate scaling factor in time could be
sufficient for the Brownian approximation to work well.
Hence, with a deterministic or Erlang service time
distribution, the approximate variances are satisfactory even if $\gamma=1.0$.
A large scaling factor is necessary if the variability of service times is considerable.
When the service time distribution is log-normal with $c_{S}^{2}=1.52$,
the approximate variances are not accurate with $\gamma=1.0$. To get
adequate approximations, the mean patience time should be at least several
times longer than the mean service time.
The approximate variances are satisfactory when $\gamma=5.0$ and~$10$.

\begin{figure}[t]
\centering
\begin{subfigure}[b]{0.45\textwidth}
\includegraphics[width=2.8in]{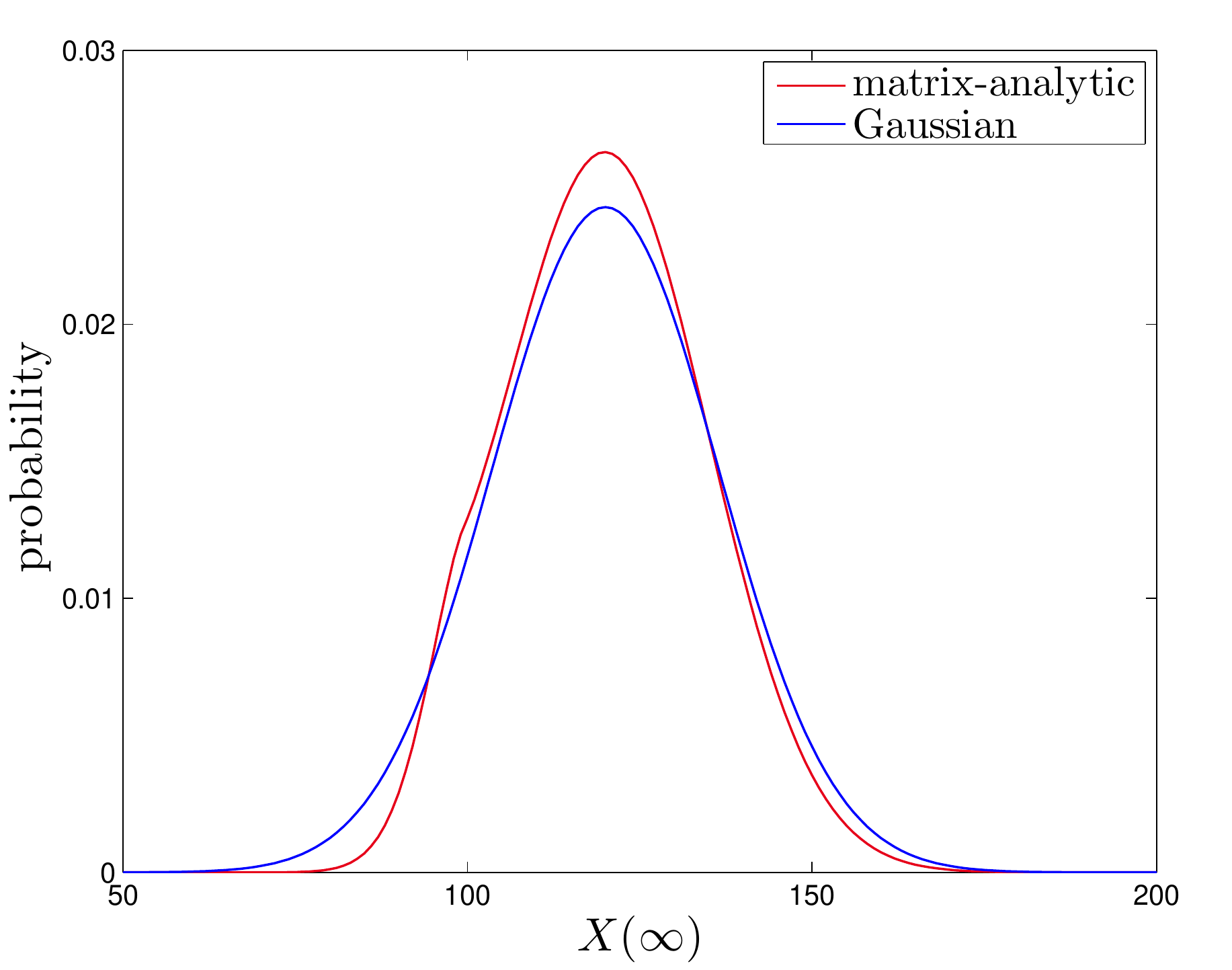}
\caption{$\gamma = 1.0$}
\label{fig:H2-gamma1}
\end{subfigure}
\begin{subfigure}[b]{0.45\textwidth}
\includegraphics[width=2.8in]{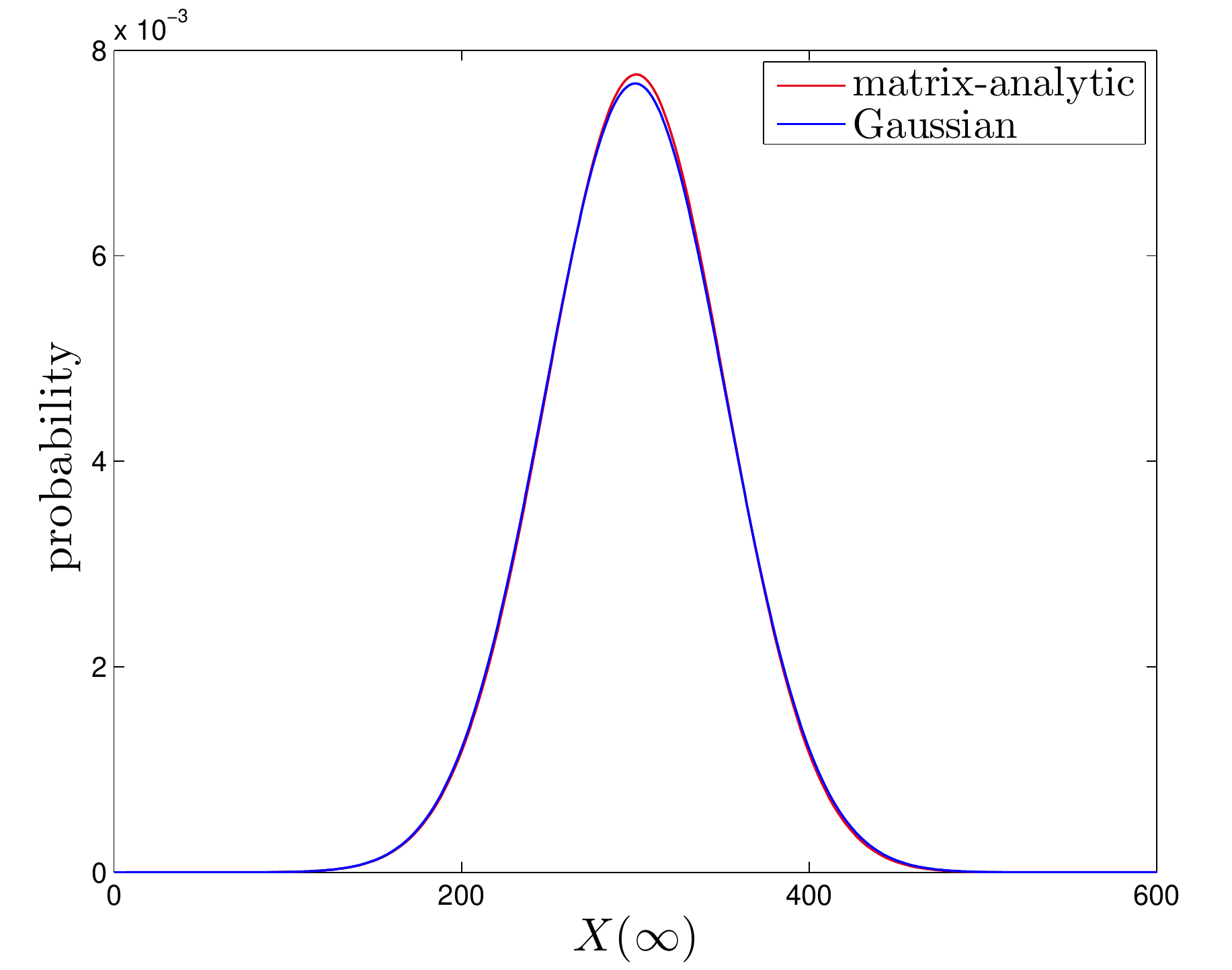}
\caption{$\gamma = 10$}
\label{fig:H2-gamma10}
\end{subfigure}
\caption{The steady-state distribution of the number of customers in an
$\mbox{M}/\mbox{H}_{2}/100+\mbox{M}$ queue with $\mu=1.0$, $\rho=1.2$,
and $c_{S}^{2}=4.0$; the exact distribution by the matrix-analytic method
is compared with the Gaussian approximation from the diffusion model.}
\label{fig:H2}
\end{figure}

To examine the steady-state queue length and virtual waiting time
distributions, we list some tail probabilities in
Table~\ref{table:tail100}. The distributions of the scaled queue
length and virtual waiting time are compared with the Gaussian distributions in
\eqref{eq:steadyQ} and \eqref{eq:steadyW}. The results in this table are
consistent with what we found in Table~\ref{table:measure100}:
With the deterministic or Erlang service time distribution,
the approximate distributions are satisfactory when the mean patience time is comparable
to or longer than the mean service time; when service times follow the
log-normal distribution that has a larger variance, the mean patience time
is required to be at least several times longer than the mean service time
for the Gaussian distributions to be accurate.

To illustrate how the scaled queue length converges to a Gaussian
random variable, let us examine an
$\mbox{M}/\mbox{H}_{2}/100+\mbox{M}$ queue that has an
hyperexponential service time distribution with $1/\mu=1.0$ and $c_{S}^{2}=4.0$.
There are two types of customers in this system. The service times
of either type are iid following an exponential
distribution. The fraction of the first type is $67.41\%$
and its mean service time is $0.1484$, and the fraction of the
second type is $32.59\%$ and its mean service time is $2.761$. The
distribution of the stationary number of customers in this system can be
computed by the matrix-analytic method (see \citet{LatoucheRamaswami99}).
By \eqref{eq:steadyQ}, we can approximate this distribution by
\[
\mathbb{P}[X(\infty)=i] \approx \frac{1}{\sqrt{n\gamma\mu(\rho c_{A}^2+c_{S}^2+\rho-1)/2}}\phi
\bigg(\frac{i-n-n\mu(\rho-1)\gamma}{\sqrt{n\gamma\mu(\rho c_{A}^2+c_{S}^2+\rho-1)/2}}\bigg)
\quad\mbox{for }i\in\mathbb{N}_{0},
\]
where $\phi$ is the standard Gaussian density function.
We compare the distribution produced by the matrix-analytic method
with the approximate distribution in Figure~\ref{fig:H2}. Although
the Gaussian approximation does not capture the exact
distribution with $\gamma=1.0$, it is a good fit with $\gamma=10$.

\begin{table}[t]
\caption{Performance estimates for an $\mbox{M}/\mbox{GI}/5+\mbox{M}$ queue
with $\mu=1.0$ and $\rho=1.2$; simulation results (with $95\%$ confidence intervals)
are compared with approximate results (in \emph{italics}).}
\centering %
\footnotesize{
\begin{tabular}{l|lllll}
\hline
\multicolumn{1}{c|}{} & \multicolumn{1}{c}{} & \multicolumn{2}{c}{Queue length} & \multicolumn{2}{c}{Virtual waiting time}\T\tabularnewline
\multicolumn{1}{l|}{Patience} & \multicolumn{1}{l}{Abd. fraction} & \multicolumn{1}{l}{Mean} & \multicolumn{1}{l}{Variance} & \multicolumn{1}{l}{Mean} & \multicolumn{1}{l}{Variance}\B\tabularnewline
\hline
\multicolumn{1}{c|}{} & \multicolumn{5}{c}{$\mbox{M}/\mbox{D}/5+\mbox{M}$}\T\B\tabularnewline
\cline{2-6}
$\gamma=5.0$  & $0.1814$  & $5.441$  & $14.93$  & $1.041$  & $0.4121$\T\tabularnewline
 & $\pm0.000094$  & $\pm0.0037$  & $\pm0.020$  & $\pm0.00066$  & $\pm0.00050$\tabularnewline
 & $\emph{0.1667}$  & $\emph{5.000}$  & $\emph{17.50}$  & $\emph{0.9116}$  & $\emph{0.5000}$\B\tabularnewline
$\gamma=20$  & $0.1672$  & $20.06$  & $68.95$  & $3.708$  & $1.9649$ \T\tabularnewline
 & $\pm0.00011$  & $\pm0.017$  & $\pm0.18$  & $\pm0.0029$  & $\pm0.0049$\tabularnewline
 & $\emph{0.1667}$  & $\emph{20.00}$  & $\emph{70.00}$  & $\emph{3.646}$  & $\emph{2.000}$\B\tabularnewline
$\gamma=50$  & $0.1666$  & $49.99$  & $175.5$  & $9.164$  & $5.017$ \T\tabularnewline
 & $\pm0.00011$  & $\pm0.043$  & $\pm0.66$  & $\pm0.0074$  & $\pm0.018$\tabularnewline
 & $\emph{0.1667}$  & $\emph{50.00}$  & $\emph{175.0}$  & $\emph{9.116}$  & $\emph{5.000}$\B\tabularnewline
\hline
\multicolumn{1}{c|}{} & \multicolumn{5}{c}{$\mbox{M}/\mbox{E}_{2}/5+\mbox{M}$}\T\B\tabularnewline
\cline{2-6}
$\gamma=5.0$  & $0.1896$  & $5.689$  & $18.66$  & $1.109$  & $0.5939$ \T\tabularnewline
 & $\pm0.00012$  & $\pm0.0044$  & $\pm0.020$  & $\pm0.00080$  & $\pm0.00065$\tabularnewline
 & $\emph{0.1667}$  & $\emph{5.000}$  & $\emph{23.75}$  & $\emph{0.9116}$  & $\emph{0.7500}$\B\tabularnewline
$\gamma=20$  & $0.1687$  & $20.25$  & $90.30$  & $3.766$  & $2.852$ \T\tabularnewline
 & $\pm0.00014$  & $\pm0.021$  & $\pm0.16$  & $\pm0.0036$  & $\pm0.0050$\tabularnewline
 & $\emph{0.1667}$  & $\emph{20.00}$  & $\emph{95.00}$  & $\emph{3.646}$  & $\emph{3.000}$\B\tabularnewline
$\gamma=50$  & $0.1668$  & $50.04$  & $237.0$  & $9.198$  & $7.491$ \T\tabularnewline
 & $\pm0.00013$  & $\pm0.053$  & $\pm0.77$  & $\pm0.0093$  & $\pm0.025$\tabularnewline
 & $\emph{0.1667}$  & $\emph{50.00}$  & $\emph{237.5}$  & $\emph{9.116}$  & $\emph{7.500}$\B\tabularnewline
\hline
\multicolumn{1}{c|}{} & \multicolumn{5}{c}{$\mbox{M}/\mbox{LN}/5+\mbox{M}$}\T\B\tabularnewline
\cline{2-6}
$\gamma=5.0$  & $0.1985$  & $5.953$  & $23.90$  & $1.191$  & $0.8979$ \T\tabularnewline
 & $\pm0.00013$  & $\pm0.0044$  & $\pm0.028$  & $\pm0.00087$  & $\pm0.0014$\tabularnewline
 & $\emph{0.1667}$  & $\emph{5.000}$  & $\emph{36.50}$  & $\emph{0.9116}$  & $\emph{1.260}$\B\tabularnewline
$\gamma=20$  & $0.1716$  & $20.59$  & $126.4$  & $3.875$  & $4.458$ \T\tabularnewline
 & $\pm0.00015$  & $\pm0.019$  & $\pm0.27$  & $\pm0.0035$  & $\pm0.010$\tabularnewline
 & $\emph{0.1667}$  & $\emph{20.00}$  & $\emph{146.0}$  & $\emph{3.646}$  & $\emph{5.040}$\B\tabularnewline
$\gamma=50$  & $0.1670$  & $50.09$  & $354.5$  & $9.258$  & $12.32$ \T\tabularnewline
 & $\pm0.00016$  & $\pm0.052$  & $\pm1.1$  & $\pm0.0095$  & $\pm0.042$\tabularnewline
 & $\emph{0.1667}$  & $\emph{50.00}$  & $\emph{365.0}$  & $\emph{9.116}$  & $\emph{12.60}$\B\tabularnewline
\hline
\end{tabular}
}
\label{table:measure5}
\end{table}

\begin{table}[t]
\caption{Tail probabilities for queue length and virtual waiting time in an $\mbox{M}/\mbox{GI}/5+\mbox{M}$ queue
with $\mu=1.0$ and $\rho=1.2$; simulation results (with $95\%$ confidence intervals)
are compared with diffusion approximations (in \emph{italics}).}
\centering %
\footnotesize{
\begin{tabular}{l|llllll}
\hline
 & \multicolumn{3}{c}{$\mathbb{P}[\tilde{X}(\infty)>a]$} & \multicolumn{3}{c}{$\mathbb{P}[\tilde{W}(\infty)>a]$}\T\tabularnewline
\multicolumn{1}{c|}{Patience} & \multicolumn{1}{c}{$a=0.5$} & \multicolumn{1}{c}{$a=1.0$} & \multicolumn{1}{c}{$a=2.0$} & \multicolumn{1}{c}{$a=0.5$} & \multicolumn{1}{c}{$a=1.0$} & \multicolumn{1}{c}{$a=2.0$}\B\tabularnewline
\hline
\multicolumn{1}{c|}{} & \multicolumn{6}{c}{$\mbox{M}/\mbox{D}/5+\mbox{M}$}\T\B\tabularnewline
\cline{2-7}
$\gamma=5.0$  & $0.2794$  & $0.1073$  & $0.01219$  & $0.2774$  & $0.09800$  & $0.003874$\T\tabularnewline
 & $\pm0.00039$  & $\pm0.00027$  & $\pm0.000090$  & $\pm0.00039$  & $\pm0.00025$  & $\pm0.000041$\B\tabularnewline
$\gamma=20$  & $0.2895$  & $0.1082$ & $0.01012$  & $0.2513$  & $0.08525$  & $0.002887$\T\tabularnewline
 & $\pm0.00080$  & $\pm0.00056$  & $\pm0.00014$  & $\pm0.00077$  & $\pm0.00050$  & $\pm0.000061$\B\tabularnewline
$\gamma=50$  & $0.2811$  & $0.1220$ & $0.01071$  & $0.2464$  & $0.08261$  & $0.002712$\T\tabularnewline
 & $\pm0.0012$  & $\pm0.00095$  & $\pm0.00020$  & $\pm0.0012$  & $\pm0.00077$  & $\pm0.000081$\B\tabularnewline
 & $\emph{0.2750}$  & $\emph{0.1160}$  & $\emph{0.008414}$  & $\emph{0.2398}$  & $\emph{0.07865}$  & $\emph{0.002339}$\T\B\tabularnewline
\hline
\multicolumn{1}{c|}{} & \multicolumn{6}{c}{$\mbox{M}/\mbox{E}_{2}/5+\mbox{M}$}\T\B\tabularnewline
\cline{2-7}
$\gamma=5.0$  & $0.3147$  & $0.1420$  & $0.02319$  & $0.3292$  & $0.1563$  & $0.01841$\T\tabularnewline
 & $\pm0.00041$  & $\pm0.00029$  & $\pm0.00010$  & $\pm0.00043$  & $\pm0.00029$  & $\pm0.000089$\B\tabularnewline
$\gamma=20$  & $0.3200$  & $0.1436$ & $0.02124$  & $0.2979$  & $0.1361$  & $0.01359$\T\tabularnewline
 & $\pm0.00084$  & $\pm0.00055$  & $\pm0.00018$  & $\pm0.00084$  & $\pm0.00058$  & $\pm0.00010$\B\tabularnewline
$\gamma=50$  & $0.3106$  & $0.1579$ & $0.02256$  & $0.2906$  & $0.1308$  & $0.01210$\T\tabularnewline
 & $\pm0.0013$  & $\pm0.00094$  & $\pm0.00026$  & $\pm0.0013$  & $\pm0.00083$  & $\pm0.00018$\B\tabularnewline
 & $\emph{0.3040}$  & $\emph{0.1525}$  & $\emph{0.02009}$  & $\emph{0.2819}$  & $\emph{0.1241}$  & $\emph{0.01046}$\T\B\tabularnewline
\hline
\multicolumn{1}{c|}{} & \multicolumn{6}{c}{$\mbox{M}/\mbox{LN}/5+\mbox{M}$}\T\B\tabularnewline
\cline{2-7}
$\gamma=5.0$  & $0.3432$  & $0.1802$  & $0.04271$  & $0.3645$  & $0.2130$  & $0.05339$\T\tabularnewline
 & $\pm0.00041$  & $\pm0.00032$  & $\pm0.00014$  & $\pm0.00039$  & $\pm0.00031$  & $\pm0.00019$\B\tabularnewline
$\gamma=20$  & $0.3522$  & $0.1905$ & $0.04674$  & $0.3411$  & $0.1983$  & $0.04691$\T\tabularnewline
 & $\pm0.00065$  & $\pm0.00055$  & $\pm0.00029$  & $\pm0.00067$  & $\pm0.00054$  & $\pm0.00029$\B\tabularnewline
$\gamma=50$  & $0.3419$  & $0.2055$ & $0.05100$  & $0.3328$  & $0.1927$  & $0.04372$\T\tabularnewline
 & $\pm0.00095$  & $\pm0.00082$  & $\pm0.00046$  & $\pm0.0011$  & $\pm0.00083$  & $\pm0.00042$\B\tabularnewline
 & $\emph{0.3395}$  & $\emph{0.2039}$  & $\emph{0.04894}$  & $\emph{0.3280}$  & $\emph{0.1865}$  & $\emph{0.03740}$\T\B\tabularnewline
\hline
\end{tabular}
}
\label{table:tail5}
\end{table}

\subsection{Examples in the long patience overloaded regime
\label{subsection:long-patience}}
Let us examine an $\mbox{M}/\mbox{GI}/5+\mbox{M}$ queue with
$\lambda=6.0$ and $\gamma = 5.0$, $20$, and $50$, respectively.
Since the mean patience time is much
longer than the mean service time, this queue is in the long
patience overloaded regime. The corresponding performance
estimates are listed in Tables~\ref{table:measure5} and~\ref{table:tail5}.
As in Section~\ref{subsection:many-server}, we obtain the
approximate results in Table~\ref{table:measure5} by
\eqref{eq:abd-fraction}, \eqref{eq:mean-queue}, \eqref{eq:var-queue},
\eqref{eq:mean-virtual}, and \eqref{eq:var-virtual}, and obtain
the approximate tail probabilities in Table~\ref{table:tail5}
by \eqref{eq:steadyQ} and \eqref{eq:steadyW}.

The diffusion model approximates a queue whose servers are
almost always busy. This condition may not hold if the traffic
intensity is not significantly greater than $1$, the queue has only one or
several servers, and the mean patience time is not very long.
With $\rho=1.2$, $n=5$, and $\gamma=5.0$, the abandonment fraction
of the queue is notably greater than the approximate fraction for
all three service time distributions. This implies that the idling
time of servers is no longer negligible. In this case,
the diffusion model may not produce adequate results.
The idling time of servers can be reduced
by increasing the mean patience time: As customers become more
patient, the queue length grows longer in the overloaded system,
which in turn prevents the servers from idling. In
Tables~\ref{table:measure5} and~\ref{table:tail5}, the approximate
results become much more accurate with $\gamma=20$ and $50$.

When an overloaded queue has one or several servers, the Brownian
approximation used in the diffusion model also requires a large
mean patience time. Note that the service completion process is close to the superposition of
$n$ renewal processes. When $n$ is a small integer, by the FCLT
for renewal processes, a large scaling factor in time is a
prerequisite for the scaled superposition process
to behave like a Brownian motion. In contrast, when $n$ is a large
integer, the scaling in space renders the superposition
process close to Gaussian (see Theorem~2 in \citet{Whitt85}).
Then, as long as the scaling in time can sufficiently
reduce the dependence of the increments to the history, the space-time scaled superposition process will be close
to a Brownian motion. A moderate scaling
factor in time is usually sufficient if the variability of service times is not
large. This contrast can be confirmed by comparing
Tables~\ref{table:measure100} and~\ref{table:tail100}
with Tables~\ref{table:measure5} and~\ref{table:tail5}:
A mean patience time that is several times longer
than the mean service time leads to satisfactory approximate results for a
queue with one hundred servers; in a queue with merely five servers,
however, the mean patience time has to be tens of times longer than
the mean patience time for the diffusion model to work well.

\section{Proof of Theorem~\ref{theorem:many-server}
\label{sec:Proof-queue}}
A sequence of perturbed systems is introduced in Section~\ref{subsec:perturbed}.
In Section~\ref{subsec:equivalence}, we first show that the perturbed systems
are asymptotically equivalent to the original queues, and then prove the
diffusion limit for the perturbed systems.

\subsection{A perturbed system
\label{subsec:perturbed}}

In the $n$th system, the number of customers at time $t$ follows
the dynamical equation
\begin{equation}
X_{n}(t)=X_{n}(0)+E_{n}(t)-A_{n}(t)-D_{n}(t)\quad\mbox{for }t\geq0,
\label{eq:system}
\end{equation}
where $A_{n}(t)$ is the number of customers who have abandoned the
system during $(0,t]$ and $D_{n}(t)$ is the number of service completions
during $(0,t]$. The abandonment process $A_{n}$ can be generated
via the following standard procedure. Let $G$ be a unit-rate Poisson
process that is independent of $X_{n}(0)$, $E_{n}$, and $N_{1},\ldots,N_{n}$
in \eqref{eq:renewal}.
Let $Q_{n}(t)$ be the queue length at time~$t$, i.e.,
\begin{equation}
Q_{n}(t)=(X_{n}(t)-n)^{+}.\label{eq:Q}
\end{equation}
Because the patience time distribution is exponential with mean $\gamma_{n}$,
the instantaneous abandonment rate at~$t$ is $\gamma_{n}^{-1}Q_{n}(t)$.
We may generate the abandonment process $A_{n}$ by
\begin{equation}
A_{n}(t)=G\Big(\gamma_{n}^{-1}\int_{0}^{t}Q_{n}(u)\,\mathrm{d}u\Big).
\label{eq:abandonment}
\end{equation}
For the departure process $D_{n}$, because $\{\xi_{j,k}:k\in\mathbb{N}\}$
is the sequence of service times to be finished by the $j$th server,
the service completion process from this server is identical to $N_{j}$
until the $j$th server begins to idle. Therefore, $D_{n}$ is identical
to the superposition of $N_{1},\ldots,N_{n}$ until the first idle
server appears. Let
\[
\tau_{n}=\inf\{t\geq0:X_{n}(t)<n\},
\]
which is the time that the first idle server appears. Because all
servers are busy at time~$0$, we have $\tau_{n}>0$. The departure
process satisfies
\begin{equation}
D_{n}(t)=B_{n}(t)\quad\mbox{for }0\leq t\leq\tau_{n},\label{eq:DB}
\end{equation}
with $B_{n}$ given by \eqref{eq:B}. As the superposition of $n$
iid stationary renewal processes, $B_{n}$ is more analytically tractable
than $D_{n}$. The equivalence between these two processes up to $\tau_{n}$
allows us to introduce a perturbed system that has simplified dynamics.
This perturbed system is asymptotically equivalent to the
original queue as $n$ goes large.

Consider the system equation \eqref{eq:system}. By \eqref{eq:Q}--\eqref{eq:DB},
\[
X_{n}(t)=X_{n}(0)+E_{n}(t)-G\Big(\gamma_{n}^{-1}\int_{0}^{t}(X_{n}(u)-n)^{+}
\,\mathrm{d}u\Big)-B_{n}(t)\quad\mbox{for }0\leq t\leq\tau_{n}.
\]
From this equation, we introduce a new process $Y_{n}$ by
\begin{equation}
Y_{n}(t) = Y_{n}(0)+E_{n}(t)-G\Big(\gamma_{n}^{-1}\int_{0}^{t}(Y_{n}(u)-n)^{+}
\,\mathrm{d}u\Big)-B_{n}(t)\quad\mbox{for }t\geq0,
\label{eq:perturb}
\end{equation}
where we set $Y_{n}(0)=X_{n}(0)$.
We refer to \eqref{eq:perturb}
as the \emph{perturbed system equation}. Clearly,
\begin{equation}
Y_{n}(t)=X_{n}(t)\quad\mbox{for }0\leq t\leq\tau_{n}
\label{eq:identical}
\end{equation}
on each sample path. Thus, $\tau_{n}$ can be defined alternatively by
\begin{equation}
\tau_{n}=\inf\{t\geq0:Y_{n}(t)<n\}.\label{eq:tau_n}
\end{equation}
The perturbed system can be envisioned as a queue where no server
is allowed to idle. If a server finds the buffer empty upon a service
completion, she begins to serve a customer who has not arrived yet.
In the perturbed system, all
servers are always busy and the departure process from each server
is a stationary renewal process.

\subsection{Limit processes for perturbed systems and asymptotic equivalence
\label{subsec:equivalence}}

We will prove Theorem~\ref{theorem:many-server} by a continuous
mapping approach where two continuous maps are involved. The first
map is used to prove a fluid limit, and the second is for a diffusion limit. The fluid
limit enables us to establish the asymptotic equivalence between the
original queues and the perturbed systems, which implies that
these two sequences of systems have the same diffusion limit.

For any $f\in\mathbb{D}$, let $x$ and $z$ be two functions in $\mathbb{D}$
such that
\begin{equation}
x(t)=f(t)-\int_{0}^{t}x(u)^{+}\,\mathrm{d}u\quad\mbox{and}\quad z(t)=f(t)-\int_{0}^{t}z(u)\,\mathrm{d}u.
\label{eq:map}
\end{equation}
By Theorem~4.1 in \citet{PangETAL07}, each integral equation defines
a continuous map.

\begin{lemma}\label{lemma:map}For each $f\in\mathbb{D}$, there
is a unique $(x,z)\in\mathbb{D}\times\mathbb{D}$ such that \eqref{eq:map}
holds. Let $\varphi:\mathbb{D}\rightarrow\mathbb{D}$ be the function
that maps $f$ to $x$ and $\psi:\mathbb{D}\rightarrow\mathbb{D}$
be the function that maps $f$ to $z$. Then, $\varphi$ and $\psi$
are continuous maps when $\mathbb{D}$ (as both the domain and the
range) is endowed with the $J_{1}$ topology.

\end{lemma}

In the fluid scaling, the perturbed system equation \eqref{eq:perturb}
can be written as
\[
\bar{Y}_{n}(t)=\bar{Y}_{n}(0)+\bar{E}_{n}(t)-\bar{G}_{n}\Big(\int_{0}^{t}\bar{Y}_{n}(u)^{+}\,\mathrm{d}u\Big)-\bar{B}_{n}(t)-\int_{0}^{t}\bar{Y}_{n}(u)^{+}\,\mathrm{d}u,
\]
where
\begin{equation}
\bar{E}_{n}(t)=\frac{1}{n\gamma_{n}}E_{n}(\gamma_{n}t),\quad\bar{G}_{n}(t)=\frac{1}{n\gamma_{n}}(G(n\gamma_{n}t)-n\gamma_{n}t),\quad\bar{B}_{n}(t)=\frac{1}{n\gamma_{n}}B_{n}(\gamma_{n}t),\label{eq:EGB-bar}
\end{equation}
and
\begin{equation}
\bar{Y}_{n}(t)=\frac{1}{n\gamma_{n}}(Y_{n}(\gamma_{n}t)-n).\label{eq:Y-bar}
\end{equation}

\begin{lemma}
\label{lemma:fluid}
Under the conditions of Theorem~\ref{theorem:many-server},
\[
\bar{Y}_{n}\Rightarrow\mu(\rho-1)\chi\quad\mbox{as }n\rightarrow\infty.
\]
\end{lemma}

\begin{proof}
By \eqref{eq:lambda_n} and \eqref{eq:E}, $\bar{E}_{n}\Rightarrow\rho\mu e$
as $n\rightarrow\infty$. Since $Y_{n}(0)=X_{n}(0)$, we have
$\bar{Y}_{n}(0)\Rightarrow(\rho-1)\mu$ as
$n\rightarrow\infty$ by \eqref{eq:initial}. Because $\bar{Y}_{n}(t)\leq\bar{Y}_{n}(0)+\bar{E}_{n}(t)$,
\begin{equation}
\lim_{a\rightarrow\infty}\limsup_{n\rightarrow\infty}
\mathbb{P\Big[}\sup_{0\leq t\leq T}\bar{Y}_{n}(t)>a\Big]=0\quad\mbox{for all }T>0.
\label{eq:Y-bar-bound}
\end{equation}
The functional law of large numbers (see Theorem~5.10 in \citet{ChenYao01})
implies that $\bar{G}_{n}\Rightarrow0$
as $n\rightarrow\infty$, which, along with \eqref{eq:Y-bar-bound}, implies that
\[
\Big\{\bar{G}_{n}\Big(\int_{0}^{t}\bar{Y}_{n}(u)^{+}\,\mathrm{d}u\Big):t\geq0\Big\}\Rightarrow0\quad\mbox{as }n\rightarrow\infty.
\]
Proposition~\ref{proposition:FSLLN} in the appendix states that
$\bar{B}_{n}\Rightarrow\mu e$ as $n\rightarrow\infty$. Put
\[
\bar{M}_{n}(t)=\bar{Y}_{n}(0)+\bar{E}_{n}(t)-\bar{G}_{n}\Big(\int_{0}^{t}\bar{Y}_{n}(u)^{+}\,\mathrm{d}u\Big)-\bar{B}_{n}(t).
\]
We deduce from the previous convergence results that
$\bar{M}_{n}\Rightarrow\mu(\rho-1)(\chi+e)$
as $n\rightarrow\infty$. Note that $\varphi(\mu(\rho-1)(\chi+e))=\mu(\rho-1)\chi$.
Because $\bar{Y}_{n}=\varphi(\bar{M}_{n})$, the fluid limit follows
from Lemma~\ref{lemma:map} and the continuous mapping theorem (see
Theorem~5.2 in \citet{ChenYao01}).
\end{proof}

Let
\begin{equation}
\bar{\tau}_{n}=\gamma_{n}^{-1}\tau_{n}.\label{eq:tau-tilde}
\end{equation}
Then, $\bar{\tau}_{n}$ is the instant when the first idle server
appears in the time-scaled system. By \eqref{eq:identical},
\begin{equation}
\tilde{X}_{n}(t)=\tilde{Y}_{n}(t)\quad\mbox{for }0\leq t\leq\bar{\tau}_{n},
\label{eq:XY}
\end{equation}
where
\[
\tilde{Y}_{n}(t) = \frac{1}{\sqrt{n\gamma_{n}}}(Y_{n}(\gamma_{n}t)-n-n\mu(\rho-1)\gamma_{n}).
\]
The next lemma states that $\bar{\tau}_{n}\rightarrow\infty$ in probability
as $n\rightarrow\infty$, which implies that $\tilde{X}_{n}$ and
$\tilde{Y}_{n}$ are asymptotically equal over any finite time interval.

\begin{lemma}
\label{lemma:tau_n}
Under the conditions of Theorem~\ref{theorem:many-server},
\[
\lim_{n\rightarrow\infty}\mathbb{P}[\bar{\tau}_{n}\leq T]=0\quad\mbox{for all }T>0\mbox{.}
\]
\end{lemma}

\begin{proof}
By \eqref{eq:tau_n}, \eqref{eq:Y-bar}, and \eqref{eq:tau-tilde},
$\bar{\tau}_{n}=\inf\{t\geq0:\bar{Y}_{n}(t)<0\}$, which yields
\[
\mathbb{P}[\bar{\tau}_{n}\leq T]=\mathbb{P}\Big[\inf_{0\leq t\leq T}\bar{Y}_{n}(t)<0\Big].
\]
Then, the assertion follows from Lemma~\ref{lemma:fluid}.
\end{proof}

Put
\begin{align*}
\tilde{A}_{n}(t) & =\frac{1}{\sqrt{n\gamma_{n}}}\Big(G\Big(\gamma_{n}^{-1}\int_{0}^{\gamma_{n}t}(Y_{n}(u)-n)^{+}\,\mathrm{d}u\Big)-\gamma_{n}^{-1}\int_{0}^{\gamma_{n}t}(Y_{n}(u)-n)^{+}\,\mathrm{d}u\Big),\\
\tilde{\Delta}_{n}(t) & =\frac{1}{\sqrt{n\gamma_{n}}}\gamma_{n}^{-1}\int_{0}^{\gamma_{n}t}(Y_{n}(u)-n)^{-}\,\mathrm{d}u.
\end{align*}
With these processes, we can derive a diffusion-scaled version of the
dynamical equation \eqref{eq:perturb},
\[
\tilde{Y}_{n}(t)=\tilde{Y}_{n}(0)+\tilde{E}_{n}(t)-\tilde{A}_{n}(t)-\tilde{\Delta}_{n}(t)-\tilde{B}_{n}(t)-\int_{0}^{t}\tilde{Y}_{n}(u)\,\mathrm{d}u.
\]

\begin{lemma} \label{lemma:Y}Under the conditions of Theorem~\ref{theorem:many-server},
\[
\tilde{Y}_{n}\Rightarrow\hat{X}\quad\mbox{as }n\rightarrow\infty.
\]
\end{lemma}

\begin{proof}
Let
\[
\tilde{M}_{n}(t)=\tilde{Y}_{n}(0)+\tilde{E}_{n}(t)-\tilde{A}_{n}(t)-\tilde{\Delta}_{n}(t)-\tilde{B}_{n}(t).
\]
Because $\tilde{Y}_{n}=\psi(\tilde{M}_{n})$ and $\hat{X}=\psi(\hat{M})$,
Lemma~\ref{lemma:map} and the continuous mapping theorem will lead
to the assertion once we prove $\tilde{M}_{n}\Rightarrow\hat{M}$ as
$n\rightarrow\infty$.

Put
\[
\tilde{G}_{n}(t)=\frac{1}{\sqrt{n\gamma_{n}}}\big(G(n\gamma_{n}(\rho-1)\mu t)-n\gamma_{n}(\rho-1)\mu t\big).
\]
By the FCLT for renewal processes, $\tilde{G}_{n}\Rightarrow\hat{A}$
as $n\rightarrow\infty$ where $\hat{A}$ is a driftless Brownian
motion with variance $(\rho-1)\mu$ and $\hat{A}(0)=0$. Recall that
$Y_{n}(0)=X_{n}(0)$ and $\tilde{Y}_{n}(0),\tilde{E}_{n},\tilde{G}_{n},\tilde{B}_{n}$
are mutually independent. By \eqref{eq:E}, \eqref{eq:initial}, and
Theorem~\ref{theorem:FCLT},
\begin{equation}
\tilde{Y}_{n}(0)+\tilde{E}_{n}-\tilde{G}_{n}-\tilde{B}_{n}\Rightarrow\hat{M}\quad\mbox{as }n\rightarrow\infty.\label{eq:M-hat}
\end{equation}
Put
\[
\bar{\zeta}_{n}(t)=\frac{1}{(\rho-1)\mu}\int_{0}^{t}\bar{Y}_{n}(u)^{+}\,\mathrm{d}u.
\]
Then, $\tilde{A}_{n}=\tilde{G}_{n}\circ\bar{\zeta}_{n}$. By Lemma~\ref{lemma:fluid},
$\bar{\zeta}_{n}\Rightarrow e$ as $n\rightarrow\infty$. Because
$\tilde{G}_{n}\Rightarrow\hat{A}$ and $\hat{A}$ has continuous paths
almost surely, it follows that
\begin{equation}
\tilde{A}_{n}-\tilde{G}_{n}\Rightarrow0\quad\mbox{as }n\rightarrow\infty.\label{eq:AG}
\end{equation}
Moreover,
\[
\mathbb{P}\Big[\sup_{0\leq t\leq T}\tilde{\Delta}_{n}(t)>0\Big]\leq\mathbb{P}\Big[\inf_{0\leq t\leq T}Y_{n}(\gamma_{n}t)<n\Big]=\mathbb{P}[\bar{\tau}_{n}\leq T]\quad\mbox{for all }T>0.
\]
Then, Lemma~\ref{lemma:tau_n} implies that
\begin{equation}
\tilde{\Delta}_{n}\Rightarrow0\quad\mbox{as }n\rightarrow\infty.\label{eq:Delta}
\end{equation}
It follows from \eqref{eq:M-hat}--\eqref{eq:Delta}
and the convergence-together theorem (see Theorem~5.4 in \citet{ChenYao01}) that
$\tilde{M}_{n}\Rightarrow\hat{M}$ as $n\rightarrow\infty$.

\end{proof}

\begin{proof}[Proof of Theorem~\ref{theorem:many-server}.]By
\eqref{eq:XY},
\[
\mathbb{P}\Big[\sup_{0\leq t\leq T}|\tilde{X}_{n}(t)-\tilde{Y}_{n}(t)|>0\Big]\leq\mathbb{P}[\bar{\tau}_{n}\leq T]\quad\mbox{for all }T>0.
\]
Lemma~\ref{lemma:tau_n} implies that $\tilde{X}_{n}-\tilde{Y}_{n}\Rightarrow0$
as $n\rightarrow\infty$. Then, the theorem follows from Lemma~\ref{lemma:Y}
and the convergence-together theorem.
\end{proof}

\section{Proof of Theorem~\ref{theorem:many-server-virtual}}
\label{sec:Proof-virtual}
The proof of Theorem~\ref{theorem:many-server-virtual} also relies
on the analysis of perturbed systems. In Section~\ref{sec:perturbed-stop},
using a perturbed system that has a stopped arrival process, we
introduce an asymptotically equivalent representation for a virtual
waiting time in the original queue. In Section~\ref{sec:limit-virtual},
we establish an asymptotic relationship between the virtual waiting time
and the queue length at a certain time in the perturbed system with
arrival stopping. We prove Theorem~\ref{theorem:many-server-virtual}
by using a diffusion limit for the queue length processes in the
perturbed systems.

\subsection{A perturbed system with a stopped arrival process
\label{sec:perturbed-stop}}

Let $s\geq0$ be a fixed number. Consider the virtual waiting time
at $s$ in the original queue. Because
the queue and its perturbed system follow the same dynamics over
$[0,\tau_{n}]$, the asymptotic equivalence proved in Lemma~\ref{lemma:tau_n}
implies an identical limit for the virtual waiting times in both
systems. We can thus explore a sequence of perturbed systems to obtain
this limit. We follow the approach adopted by \citet{TalrejaWhitt09},
exploiting a sequence of systems with stopped arrival processes.

Suppose that in the queue, the arrival process is ``turned off''
at time~$s$, i.e., all customers who arrive after $s$ are rejected.
For each $t\geq0,$ let $X_{n}^{s}(t)$ be the number of customers
at $t$. Then, $W_{n}(s)$ is the amount of time from $s$ until an
idle server appears, i.e.,
\begin{equation}
W_{n}(s)=\inf\{u\geq0:X_{n}^{s}(s+u)<n\}.\label{eq:W}
\end{equation}
In such a system with the arrival process stopping at $s$, the number
of customers at $t$ is given by
\begin{equation}
X_{n}^{s}(t)=X_{n}(0)+E_{n}^{s}(t)-A_{n}^{s}(t)-D_{n}^{s}(t),\label{eq:two-parameter-dynamics}
\end{equation}
where $E_{n}^{s}(t)=E_{n}(s\wedge t)$, $A_{n}^{s}(t)$ is the number
of abandonments by $t$, and $D_{n}^{s}(t)$ is the number of service
completions by $t$. Let $G^{s}$ be a unit-rate Poisson process that
is independent of $X_{n}(0),$ $E_{n}$, and $N_{1},\ldots,N_{n}$.
We may generate the abandonment process $A_{n}^{s}$ by
\[
A_{n}^{s}(t)=G^{s}\Big(\gamma_{n}^{-1}\int_{0}^{t}(X_{n}^{s}(u)-n)^{+}\,\mathrm{d}u\Big)\quad\mbox{for }t\geq0.
\]
(We only consider the case that $s$ is fixed, so that $G^{s}$
is allowed to change with $s$.) Because $D_{n}^{s}(t)=B_{n}(t)$ for
$0\leq s\leq\tau_{n}$ and $0\leq t\leq s+W_{n}(s)$,
the dynamical equation \eqref{eq:two-parameter-dynamics} can be written as
\[
X_{n}^{s}(t)=X_{n}(0)+E_{n}^{s}(t)-G^{s}\Big(\gamma_{n}^{-1}\int_{0}^{t}(X_{n}^{s}(u)-n)^{+}\,\mathrm{d}u\Big)-B_{n}(t)
\]
for $0\leq s\leq\tau_{n}$ and $0\leq t\leq s+W_{n}(s)$. By this
equation, we can define a process $Y_{n}^{s}$ by
\begin{equation}
Y_{n}^{s}(t)=X_{n}(0)+E_{n}^{s}(t)-G^{s}\Big(\gamma_{n}^{-1}\int_{0}^{t}(Y_{n}^{s}(u)-n)^{+}\,\mathrm{d}u\Big)-B_{n}(t)\quad\mbox{for }t\geq0.\label{eq:stopped-perturbed}
\end{equation}
Equation~\eqref{eq:stopped-perturbed} is the dynamical equation
for the $n$th perturbed system with the arrival process stopping
at $s$. Clearly,
\begin{equation}
Y_{n}^{s}(t)=X_{n}^{s}(t)\quad\mbox{for }0\leq s\leq\tau_{n}\mbox{ and }0\leq t\leq s+W_{n}(s).\label{eq:stopped-XY}
\end{equation}
Let
\begin{equation}
V_{n}(s)=\inf\{u\geq0:Y_{n}^{s}(s+u)<n\}\quad\mbox{for }s\geq0.\label{eq:Vn}
\end{equation}
Then, by \eqref{eq:W} and \eqref{eq:stopped-XY},
\begin{equation}
V_{n}(s)=W_{n}(s)\quad\mbox{for }0\leq s\leq\tau_{n}.
\label{eq:VW_equivalence}
\end{equation}

\subsection{Limit processes for perturbed systems with arrival stopping
\label{sec:limit-virtual}}

Following a continuous mapping approach, we first prove
a fluid limit for the perturbed systems with arrival stopping. Using
\eqref{eq:EGB-bar}, we can derive a fluid-scaled version of \eqref{eq:stopped-perturbed},
given by
\begin{equation}
\bar{Y}_{n}^{s}(t)=\bar{Y}_{n}(0)+\bar{E}_{n}^{s}(t)-\bar{G}_{n}^{s}\Big(\int_{0}^{t}\bar{Y}_{n}^{s}(u)^{+}\,\mathrm{d}u\Big)-\bar{B}_{n}(t)-\int_{0}^{t}\bar{Y}_{n}^{s}(u)^{+}\,\mathrm{d}u,\label{eq:stopped-fluid}
\end{equation}
where
\begin{equation}
\bar{Y}_{n}^{s}(t)=\frac{1}{n\gamma_{n}}(Y_{n}^{\gamma_{n}s}(\gamma_{n}t)-n),\quad\bar{E}_{n}^{s}(t)=\frac{1}{n\gamma_{n}}E_{n}^{\gamma_{n}s}(\gamma_{n}t),\quad\bar{G}_{n}^{s}(t)=\frac{1}{n\gamma_{n}}(G_{n}^{\gamma_{n}s}(n\gamma_{n}t)-n\gamma_{n}t).\label{eq:fluid-t}
\end{equation}

\begin{lemma}\label{Lemma:Yt-bar}Under the conditions of Theorem~\ref{theorem:many-server-virtual},
for all $s\geq0$,
\[
\bar{Y}_{n}^{s}\Rightarrow y^{s}\quad\mbox{as }n\rightarrow\infty,
\]
where $y^{s}$ is the function given by \eqref{eq:ys}.
\end{lemma}

\begin{proof}
Write
\[
\bar{M}_{n}^{s}(t)=\bar{Y}_{n}(0)+\bar{E}_{n}^{s}(t)-\bar{G}_{n}^{s}\Big(\int_{0}^{t}\bar{Y}_{n}^{s}(u)^{+}\,\mathrm{d}u\Big)-\bar{B}_{n}(t).
\]
Following the proof of Lemma~\ref{lemma:fluid}, we obtain $\bar{Y}_{n}(0)\Rightarrow(\rho-1)\mu$,
$\bar{E}_{n}^{s}\Rightarrow\rho\mu e^{s}$, $\bar{B}_{n}\Rightarrow\mu e$,
and
\[
\Big\{\bar{G}_{n}^{s}\Big(\int_{0}^{t}\bar{Y}_{n}^{s}(u)^{+}\,\mathrm{d}u\Big):t\geq0\Big\}\Rightarrow0\quad\mbox{as }n\rightarrow\infty.
\]
Then, $\bar{M}_{n}^{s}\Rightarrow\mu(\rho e^{s}-e+(\rho-1)\chi)$
as $n\rightarrow\infty$. Because $y^{s}=\varphi(\mu(\rho e^{s}-e+(\rho-1)\chi))$
and $\bar{Y}_{n}^{s}=\varphi(\bar{M}_{n}^{s})$, the fluid limit follows
from Lemma~\ref{lemma:map} and the continuous mapping theorem.

\end{proof}

Let
\[
\bar{V}_{n}(s)=\gamma_{n}^{-1}V_{n}(\gamma_{n}s),
\]
which is the virtual waiting time in the time-scaled perturbed system. By \eqref{eq:Vn}
and \eqref{eq:fluid-t},
\begin{equation}
\bar{V}_{n}(s)=\inf\{u\geq0:\bar{Y}_{n}^{s}(s+u)<0\}.\label{eq:V-bar}
\end{equation}

\begin{lemma}\label{lemma:Vn}Under the conditions of Theorem~\ref{theorem:many-server-virtual},
for all $s\geq0$,
\[
\bar{V}_{n}(s)\Rightarrow\log\rho\quad\mbox{as }n\rightarrow\infty.
\]

\end{lemma}

\begin{proof}
Because $y^{s}(s+\log\rho-\delta)>0$ and $y^{s}(s+\log\rho+\delta)<0$ for
$\delta>0$, Lemma~\ref{Lemma:Yt-bar} implies that
\[
\lim_{n\rightarrow\infty}\mathbb{P}[\bar{Y}_{n}^{s}(s+\log\rho-\delta)>0]=1\quad\mbox{and}
\quad\lim_{n\rightarrow\infty}\mathbb{P}[\bar{Y}_{n}^{s}(s+\log\rho+\delta)<0]=1.
\]
Using \eqref{eq:V-bar} and the fact that $\bar{Y}_{n}^{s}(t)$ is nonincreasing for $t\geq s$,
we obtain
\[
\lim_{n\rightarrow\infty}\mathbb{P}[\log\rho-\delta\leq\bar{V}_{n}(s)\leq\log\rho+\delta]=1,
\]
which completes the proof.

\end{proof}

Having established the convergence results in the fluid scaling, let
us turn to diffusion-scaled processes. For $t\geq0$, put
\[
\tilde{M}_{n}^{s}(t)=\tilde{Y}_{n}(0)+\tilde{E}_{n}^{s}(t)-\tilde{G}_{n}^{s}\Big(\int_{0}^{t}\bar{Y}_{n}^{s}(u)^{+}\,\mathrm{d}u\Big)-\tilde{B}_{n}(t),\label{eq:Ms-tilde}
\]
where
\[
\tilde{E}_{n}^{s}(t)=\tilde{E}_{n}(s\wedge t)\quad\mbox{and}\quad\tilde{G}{}_{n}^{s}(t)=\frac{1}{\sqrt{n\gamma_{n}}}(G_{n}^{\gamma_{n}s}(n\gamma_{n}t)-n\gamma_{n}t).
\]
\begin{lemma}\label{lemma:Ms-tilde} Let $\hat{E}^{s}(t)=\hat{E}(s\wedge t)$
and $\hat{G}^{s}$ be a standard Brownian motion independent of $\hat{X}(0)$,
$\hat{E}^{s}$, and $\hat{B}$. Under the conditions of Theorem~\ref{theorem:many-server-virtual},
for all $s\geq0$,
\[
\tilde{M}_{n}^{s}\Rightarrow\hat{M}^{s}\quad\mbox{as }n\rightarrow\infty,
\]
where
\[
\hat{M}^{s}(t)=\hat{X}(0)+\hat{E}^{s}(t)-\hat{G}^{s}\Big(\int_{0}^{t}y^{s}(u)^{+}\,\mathrm{d}u\Big)-\hat{B}(t).
\]

\end{lemma}

\begin{proof}
By \eqref{eq:E}, $\tilde{E}_{n}^{s}\Rightarrow\hat{E}^{s}$
as $n\rightarrow\infty$. By the FCLT for renewal processes, Lemma~\ref{Lemma:Yt-bar},
and the random-time-change theorem (see Theorem~5.3 in \citet{ChenYao01}),
\[
\Big\{\tilde{G}_{n}^{s}\Big(\int_{0}^{t}\bar{Y}_{n}^{s}(u)^{+}\,\mathrm{d}u\Big):t\geq0\Big\}\Rightarrow
\Big\{\hat{G}^{s}\Big(\int_{0}^{t}y^{s}(u)\,\mathrm{d}u\Big):t\geq 0\Big\}\quad\mbox{as }n\rightarrow\infty.
\]
Then, the lemma follows from \eqref{eq:initial} and Theorem~\ref{theorem:FCLT}.

\end{proof}

Now consider the diffusion-scaled queue length process, which is defined
by
\begin{align}
\tilde{Y}_{n}^{s}(t) & =\frac{1}{\sqrt{n\gamma_{n}}}(Y_{n}^{\gamma_{n}s}(\gamma_{n}t)-n-n\gamma_{n}y^{s}(t))\quad\mbox{for }0\leq t\leq s+\log\rho.\label{eq:stopped-diffusion}
\end{align}
In the subsequent proofs, $\tilde{Y}_{n}^{s}$ is considered only up to time
$s+\log\rho$. For our convenience, we set
\[
\tilde{Y}_{n}^{s}(t)=\tilde{Y}_{n}^{s}(s+\log\rho)\quad\mbox{for }t>s+\log\rho.
\]
Using these processes, we can derive the diffusion-scaled
dynamical equation from \eqref{eq:stopped-perturbed},
\begin{equation}
\tilde{Y}_{n}^{s}(t)=\tilde{M}_{n}^{s}(t)-\sqrt{n\gamma_{n}}\int_{0}^{t}(\bar{Y}_{n}^{s}(u)^{+}-y^{s}(u))\,\mathrm{d}u
\quad\mbox{for }0\leq t\leq s+\log\rho.
\label{eq:Yt_tilde}
\end{equation}
We will see that the diffusion-scaled virtual waiting time at $s$
is closely related to the diffusion-scaled queue length at $s+\log\rho$.
The next lemma is a technical result, which states the stochastic
boundedness of $\{\tilde{Y}_{n}^{s}(s+\log\rho):n\in\mathbb{N}\}$.

\begin{lemma}
\label{lemma:Yw-bounded}
Under the conditions of Theorem~\ref{theorem:many-server-virtual},
\[
\lim_{a\rightarrow\infty}\limsup_{n\rightarrow\infty}\mathbb{P}[|\tilde{Y}_{n}^{s}(s+\log\rho)|>a]=0.
\]
\end{lemma}

\begin{proof}
Because $y^{s}(u)\geq0$ for $0\leq u\leq s+\log\rho$,
\[
\sqrt{n\gamma_{n}}|\bar{Y}_{n}^{s}(u)^{+}-y^{s}(u)|\leq\sqrt{n\gamma_{n}}|\bar{Y}_{n}^{s}(u)-y^{s}(u)|=|\tilde{Y}_{n}^{s}(u)|.
\]
$^{^{^{^ {}}}}$Then by \eqref{eq:Yt_tilde},
\[
|\tilde{Y}_{n}^{s}(t)|\leq|\tilde{M}_{n}^{s}(t)|+\int_{0}^{t}|\tilde{Y}_{n}^{s}(u)|\,\mathrm{d}u\quad\mbox{for }0\leq t\leq s+\log\rho.
\]
It follows from Gronwall's inequality (see Lemma~21.4 in \citet{Kallenberg02}) that
\[
|\tilde{Y}_{n}^{s}(s+\log\rho)|\leq\sup_{0\leq t\leq s+\log\rho}|\tilde{M}_{n}^{s}(t)|\rho\exp(s).
\]
Lemma~\ref{lemma:Ms-tilde} implies that $\{\tilde{M}_{n}^{s}:n\in\mathbb{N}\}$ is stochastically
bounded. So is $\{\tilde{Y}_{n}^{s}(s+\log\rho):n\in\mathbb{N}\}$.

\end{proof}

Let
\begin{equation}
\tilde{V}_{n}(s)=\sqrt{n\gamma_{n}}(\bar{V}_{n}(s)-\log\rho),
\label{eq:Vn-tilde}
\end{equation}
which is the diffusion-scaled virtual waiting time in the perturbed
system at $s$. The following lemma states that $\tilde{V}_{n}(s)$
and $\tilde{Y}_{n}^{s}(s+\log\rho)/\mu$ are asymptotically close.

\begin{lemma}
\label{lemma:VY}
Under the conditions of Theorem~\ref{theorem:many-server-virtual},
for all $s\geq0$
\[
\tilde{V}_{n}(s)-\mu^{-1}\tilde{Y}_{n}^{s}(s+\log\rho)\Rightarrow0\quad\mbox{as }n\rightarrow\infty.
\]

\end{lemma}

\begin{proof}
Because $\bar{E}_{n}^{s}(s+\bar{V}_{n}(s))=\bar{E}_{n}^{s}(s+\log\rho)=\bar{E}_{n}(s)$,
it follows from \eqref{eq:stopped-fluid} that
\begin{multline}
\bar{Y}_{n}^{s}(s+\log\rho)=\bar{Y}_{n}^{s}(s+\bar{V}_{n}(s))+\bar{B}_{n}(s+\bar{V}_{n}(s))-\bar{B}_{n}(s+\log\rho)\\
+\bar{G}_{n}^{s}\Big(\int_{0}^{s+\bar{V}_{n}(s)}\bar{Y}_{n}^{s}(u)^{+}\,\mathrm{d}u\Big)-\bar{G}_{n}^{s}\Big(\int_{0}^{s+\log\rho}\bar{Y}_{n}^{s}(u)^{+}\,\mathrm{d}u\Big)\\
+\int_{0}^{s+\bar{V}_{n}(s)}\bar{Y}_{n}^{s}(u)^{+}\,\mathrm{d}u-\int_{0}^{s+\log\rho}\bar{Y}_{n}^{s}(u)^{+}\,\mathrm{d}u.\label{eq:VY}
\end{multline}
Multiply both sides of \eqref{eq:VY} by $\sqrt{n\gamma_{n}}$ and
let us consider each term.

By \eqref{eq:stopped-diffusion} and the fact that $y^{s}(s+\log\rho)=0$,
the left side turns out to be
\begin{equation}
\sqrt{n\gamma_{n}}\bar{Y}_{n}^{s}(s+\log\rho)=\tilde{Y}_{n}^{s}(s+\log\rho).\label{eq:left}
\end{equation}
Consider the right side. If the arrival process stops at time
$\gamma_{n}s$ for $0\leq s< \bar{\tau}_{n}$, the first idle server will appear
at $\gamma_{n}(s+\bar{V}_{n}(s))$.
This must be triggered by a service completion. Because $B_{n}$ is
the superposition of $n$ iid stationary renewal processes, the probability
that $B_{n}$ has a jump of size larger than $1$ is $0$, which implies that
\[
\mathbb{P}\Big[\bar{Y}_{n}^{s}(s+\bar{V}_{n}(s)) < -\frac{1}{n\gamma_{n}}\Big]
\leq \mathbb{P}[\bar{\tau}_{n}\leq s].
\]
Then, by Lemma~\ref{lemma:tau_n},
\begin{equation}
\sqrt{n\gamma_{n}}\bar{Y}_{n}^{s}(s+\bar{V}_{n}(s))
\Rightarrow0\quad\mbox{as }n\rightarrow\infty.\label{eq:YV}
\end{equation}
By \eqref{eq:B_tilde}, \eqref{eq:EGB-bar}, and \eqref{eq:Vn-tilde},
\[
\sqrt{n\gamma_{n}}(\bar{B}_{n}(s+\bar{V}_{n}(s))-\bar{B}_{n}(s+\log\rho))=
\tilde{B}_{n}(s+\bar{V}_{n}(s))-\tilde{B}_{n}(s+\log\rho)+\mu\tilde{V}_{n}(s),
\]
in which we have
\begin{equation}
\tilde{B}_{n}(s+\bar{V}_{n}(s))-\tilde{B}_{n}(s+\log\rho)\Rightarrow0
\quad\mbox{as }n\rightarrow\infty\label{eq:BV}
\end{equation}
by Theorem~\ref{theorem:FCLT} and Lemma~\ref{lemma:Vn}. Because
$\sqrt{n\gamma_{n}}\bar{G}_{n}^{s}=\tilde{G}_{n}^{s}$ and
$\tilde{G}_{n}^{s}\Rightarrow\hat{G}^{s}$
as $n\rightarrow\infty$, it follows from Lemmas~\ref{Lemma:Yt-bar}
and~\ref{lemma:Vn} that
\begin{equation}
\sqrt{n\gamma_{n}}\bigg(\bar{G}_{n}^{s}
\Big(\int_{0}^{s+\bar{V}_{n}(s)}\bar{Y}_{n}^{s}(u)^{+}\,\mathrm{d}u\Big)-\bar{G}_{n}^{s}\Big(\int_{0}^{s+\log\rho}\bar{Y}_{n}^{s}(u)^{+}\,\mathrm{d}u\Big)\bigg)\Rightarrow0\quad\mbox{as }n\rightarrow\infty.\label{eq:GV}
\end{equation}
Because $\bar{Y}_{n}^{s}(t)$ is nonincreasing for $t\geq s$,
\begin{multline*}
\sqrt{n\gamma_{n}}\bigg\vert\int_{0}^{s+\bar{V}_{n}(s)}\bar{Y}_{n}^{s}(u)^{+}\,\mathrm{d}u-\int_{0}^{s+\log\rho}\bar{Y}_{n}^{s}(u)^{+}\,\mathrm{d}u\bigg\vert\\
\leq|\bar{V}_{n}(s)-\log\rho||\tilde{Y}_{n}^{s}(s+\log\rho)|+|\bar{V}_{n}(s)-\log\rho|\sqrt{n\gamma_{n}}|\bar{Y}_{n}^{s}(s+\bar{V}_{n}(s))|.
\end{multline*}
Then, by \eqref{eq:YV} and Lemmas~\ref{lemma:Vn} and \ref{lemma:Yw-bounded},
\begin{equation}
\sqrt{n\gamma_{n}}\Big(\int_{0}^{s+\bar{V}_{n}(s)}\bar{Y}_{n}^{s}(u)^{+}\,\mathrm{d}u-\int_{0}^{s+\log\rho}\bar{Y}_{n}^{s}(u)^{+}\,\mathrm{d}u\Big)\Rightarrow0\quad\mbox{as }n\rightarrow\infty.\label{eq:IYV}
\end{equation}
We deduce from \eqref{eq:VY}--\eqref{eq:IYV} that $\tilde{V}_{n}(s)-\tilde{Y}_{n}^{s}(s+\log\rho)/\mu\Rightarrow0$
as $n\rightarrow\infty$.

\end{proof}

\begin{lemma}
\label{lemma:YY}
Under the conditions of Theorem~\ref{theorem:many-server-virtual},
for all $s\geq0$,
\[
\tilde{Y}_{n}^{s}\Rightarrow\hat{Y}^{s}\quad\mbox{as }n\rightarrow\infty,
\]
where
\[
\hat{Y}^{s}(t)= \hat{M}^{s}(t)-\int_{0}^{t}\hat{Y}^{s}(u)\,\mathrm{d}u \quad \mbox{for }0\leq t\leq s+\log\rho
\]
and $\hat{Y}^{s}(t)=\hat{Y}^{s}(s+\log\rho)$ for $t>s+\log\rho$.
\end{lemma}

\begin{proof}
Write
\[
\check{M}_{n}^{s}(t)=\tilde{M}_{n}^{s}(t)-\sqrt{n\gamma_{n}}\int_{0}^{t}\bar{Y}_{n}^{s}(u)^{-}\,\mathrm{d}u.
\]
By \eqref{eq:Yt_tilde}, $\tilde{Y}_{n}^{s}(t)=\psi(\check{M}_{n}^{s})(t)$
for $0\leq t\leq s+\log\rho$. If we can prove that
\begin{equation}
\sqrt{n\gamma_{n}}\int_{0}^{s+\log\rho}\bar{Y}_{n}^{s}(u)^{-}\,\mathrm{d}u
\Rightarrow0\quad\mbox{as }n\rightarrow\infty,\label{eq:Y-bar-minus}
\end{equation}
then $\check{M}_{n}^{s}\Rightarrow\hat{M}^{s}$ as $n\rightarrow\infty$
by Lemma~\ref{lemma:Ms-tilde} and the convergence-together theorem.
The current lemma will follow from
Lemma~\ref{lemma:map} and the continuous mapping theorem.

Because $\bar{Y}_{n}^{s}(t)=\bar{Y}_{n}(t)$ for $0\leq t\leq s$,
Lemma~\ref{lemma:tau_n} implies that
\[
\lim_{n\rightarrow\infty}\mathbb{P}\Big[\inf_{0\leq t\leq s}\bar{Y}_{n}^{s}(t)<0\Big]=0.
\]
Hence,
\begin{equation}
\sqrt{n\gamma_{n}}\int_{0}^{s}\bar{Y}_{n}^{s}(u)^{-}\,\mathrm{d}u\Rightarrow0\quad\mbox{as }n\rightarrow\infty.\label{eq:0S}
\end{equation}
Note that $\bar{Y}_{n}^{s}(t)$ is nonincreasing for $s\leq t \leq s+\log\rho$.
By \eqref{eq:V-bar} and \eqref{eq:stopped-diffusion},
\[
\sqrt{n\gamma_{n}}\int_{s}^{s+\log\rho}\bar{Y}_{n}^{s}(u)^{-}\,\mathrm{d}u =\sqrt{n\gamma_{n}}\int_{s+\bar{V}_{n}(s)}^{s+(\bar{V}_{n}(s)\vee\log\rho)}\bar{Y}_{n}^{s}(u)^{-}\,\mathrm{d}u
 \leq|(\bar{V}_{n}(s)-\log\rho)\tilde{Y}_{n}^{s}(s+\log\rho)|.
\]
Then, using Lemmas~\ref{lemma:Vn} and~\ref{lemma:Yw-bounded},
we have
\begin{equation}
\sqrt{n\gamma_{n}}\int_{s}^{s+\log\rho}\bar{Y}_{n}^{s}(u)^{-}\,\mathrm{d}u\Rightarrow0\quad\mbox{as }n\rightarrow\infty.\label{eq:SSrho}
\end{equation}
We obtain \eqref{eq:Y-bar-minus} by combining \eqref{eq:0S} and
\eqref{eq:SSrho}.

\end{proof}

\begin{proof}[Proof of Theorem~\ref{theorem:many-server-virtual}.]
Lemmas~\ref{lemma:VY} and~\ref{lemma:YY}, along with the
convergence-together theorem, imply that $\tilde{V}_{n}(s)\Rightarrow
\hat{Y}^{s}(s+\log\rho)/\mu$ as $n\rightarrow \infty$. The convergence of
$\tilde{W}_{n}$ follows from the asymptotic equivalence between $\tilde{W}_{n}(s)$
and $\tilde{V}_{n}(s)$, which can be deduced by \eqref{eq:VW_equivalence}
and Lemma~\ref{lemma:tau_n}. Solution \eqref{eq:Ystop} can be obtained by
Proposition~21.2 in \citet{Kallenberg02}.

\end{proof}

\section{Future work}
\label{sec:conclusion}

We have demonstrated that in two overloaded regimes, the queue
length process of a $\mbox{GI}/\mbox{GI}/n+\mbox{M}$ queue can
be approximated by an OU process. One may raise the following
questions about this diffusion model: Is the exponential patience
time distribution essential for an overloaded queue to have a
simple approximate model? With more practical patience time assumptions,
can we still approximate the steady-state queue length and the
steady-state virtual waiting time by Gaussian random variables?

We will answer these question in our subsequent work. To illustrate this,
let us consider a
$\mbox{GI}/\mbox{GI}/n+\mbox{GI}$ queue. For call center operations,
it is reasonable to assume patience times to be iid since the
waiting line is usually invisible to customers. The
$\mbox{GI}/\mbox{GI}/n+\mbox{GI}$ queue is thus an important
building block for modeling call centers. \citet{Whitt06} obtained
the mean queue length and the mean virtual waiting time of this
queue by the fluid model. Let $H$ be the distribution function
of patience times. Assume that $H$ is absolutely continuous with
density $f_{H}$. The hazard rate function of $H$ is given by
\[
h(t) = \frac{f_{H}(t)}{1-H(t)}\quad\mbox{for }t\geq 0.
\]
Let $w$ be the mean virtual waiting time. Then, $H(w)$ is the fraction
of patience times that are less than $w$. This fraction should be
approximately equal to the abandonment fraction $(\rho-1)/\rho$.
Hence, the mean virtual waiting time can be obtained by solving
\[
H(w) = \frac{\rho-1}{\rho}.
\]
For $0<s<w$, the probability that a customer who arrived $s$ time units
ago is still in the buffer is around $1-H(s)$. This implies that the
mean queue length can be approximated by
\[
q = \int_{0}^{w} \lambda(1-H(s))\,\mathrm{d}s.
\label{eq:general-q}
\]
See \citet{Whitt06} for more details. In the steady state, the virtual
waiting time process and the queue length process fluctuate around $w$ and $q$, respectively.

Some observations on queues with an exponential patience time
distribution may help us in generalizing the diffusion model.
When either $n$ or $\gamma$ is large, it follows from
\eqref{eq:mean-virtual} and \eqref{eq:var-virtual} that the standard
deviation of the virtual waiting time is much smaller
than the mean. If this condition holds with a general patience time
assumption, the abandonment process will depend on the patience time
distribution mostly through a small neighborhood of $w$. As a
consequence, the scaled queue length process will be dictated by the
patience time hazard rate at $w$. In this case, we use $\gamma = 1/h(w)$
as the scaling factor in time. If the patience time hazard rate changes
slowly around $w$, with $\gamma=1/h(w)$, we may still use \eqref{eq:steadyQ},
\eqref{eq:var-queue}, \eqref{eq:steadyW}, and \eqref{eq:var-virtual}
to approximate the steady-state distributions and variances. In particular,
it was reported in \citet{MandelbaumZeltyn13} that the patience time hazard
rate in a large call center was nearly constant after the first several
seconds of waiting (see Figure~2 in their paper). We
expect that with the above modification, the approximate formulas
are useful in performance analysis for
such a call center. If the hazard rate changes rapidly around $w$, we
may exploit the approximation scheme in \citet{ReedWard08} and \cite{ReedTezcan12}
to include the hazard rate function on a neighborhood  of
$w$ in the diffusion model. The resulting performance approximations would
be more complex, but may still have closed-form formulas. To justify the
diffusion model for queues with a general patience time
distribution, we will modify the current asymptotic regimes to incorporate
the hazard rate function. More specifically, we will combine the space-time scaling
with the hazard rate scaling proposed by \citet{ReedWard08}
in the new asymptotic framework.

\appendix
\section*{Appendix: Proof of Theorem~\ref{theorem:FCLT}\label{section:ProofFCLT}}

\setcounter{section}{1}
\setcounter{equation}{0}

Let
\begin{equation}
S_{j,k}=\sum_{\ell=1}^{k}\xi_{j,\ell}\label{eq:partial-sum}
\end{equation}
be the $k$th partial sum of $\{\xi_{j,\ell}:\ell\in\mathbb{N}\}$.
Take $S_{j,0}=0$ by convention. We first prove a functional strong
law of large numbers for the superposition of time-scaled renewal
processes.

\begin{proposition} \label{proposition:FSLLN} Let
\begin{equation}
\bar{B}_{n}(t)=\frac{1}{n\gamma_{n}}\sum_{j=1}^{n}N_{j}(\gamma_{n}t)\quad\mbox{for }t\geq0.\label{eq:B_bar}
\end{equation}
Under the conditions of Theorem~\ref{theorem:FCLT},
\[
\bar{B}_{n}\overset{\text{a.s.}}{\rightarrow}\mu e\quad\mbox{as }n\rightarrow\infty\mbox{.}
\]

\end{proposition}

\begin{proof}
Since $S_{j,N_{j}(t)}\leq t\leq S_{j,N_{j}(t)+1}$
for $t>0$, then
\[
\frac{\sum_{j=1}^{n}S_{j,N_{j}(\gamma_{n}t)}}{\sum_{j=1}^{n}N_{j}(\gamma_{n}t)}\leq\frac{n\gamma_{n}t}{\sum_{j=1}^{n}N_{j}(\gamma_{n}t)}\leq\frac{\sum_{j=1}^{n}S_{j,N_{j}(\gamma_{n}t)+1}}{\sum_{j=1}^{n}N_{j}(\gamma_{n}t)}
\]
provided that $\sum_{j=1}^{n}N_{j}(\gamma_{n}t)>0$. Note that
\[
\sum_{j=1}^{n}S_{j,N_{j}(\gamma_{n}t)+1}=\sum_{j=1}^{n}\xi_{j,1}+\sum_{j=1}^{n}\sum_{k=2}^{N_{j}(\gamma_{n}t)+1}\xi_{j,k}.
\]
Because $N_{j}(\gamma_{n}t)\overset{\text{a.s.}}{\rightarrow}\infty$
as $n\rightarrow\infty$ for $t>0$, then
\[
\frac{\sum_{j=1}^{n}\sum_{k=2}^{N_{j}(\gamma_{n}t)+1}\xi_{j,k}}{\sum_{j=1}^{n}N_{j}(\gamma_{n}t)}\overset{\text{a.s.}}{\rightarrow}\mu^{-1}\quad\mbox{as }n\rightarrow\infty
\]
by the strong law of large numbers. In addition, $n^{-1}\sum_{j=1}^{n}N_{j}(\gamma_{n}t)\overset{\text{a.s.}}{\rightarrow}\infty$
as $n\rightarrow\infty$ for $t>0$, which implies that
\[
\frac{\sum_{j=1}^{n}\xi_{j,1}}{\sum_{j=1}^{n}N_{j}(\gamma_{n}t)}\overset{\text{a.s.}}{\rightarrow}0\quad\mbox{as }n\rightarrow\infty\mbox{.}
\]
Therefore,
\[
\frac{\sum_{j=1}^{n}S_{j,N_{j}(\gamma_{n}t)+1}}{\sum_{j=1}^{n}N_{j}(\gamma_{n}t)}\overset{\text{a.s.}}{\rightarrow}\mu^{-1}\quad\mbox{as }n\rightarrow\infty\mbox{.}
\]
Also,
\[
\frac{\sum_{j=1}^{n}S_{j,N_{j}(\gamma_{n}t)}}{\sum_{j=1}^{n}N_{j}(\gamma_{n}t)}=\frac{\sum_{j=1}^{n}S_{j,N_{j}(\gamma_{n}t)}}{\sum_{j=1}^{n}(N_{j}(\gamma_{n}t)-1)}\cdot\frac{\sum_{j=1}^{n}(N_{j}(\gamma_{n}t)-1)}{\sum_{j=1}^{n}N_{j}(\gamma_{n}t)}\overset{\text{a.s.}}{\rightarrow}\mu^{-1}\quad\mbox{as }n\rightarrow\infty\mbox{.}
\]
Then, $\bar{B}_{n}(t)\overset{\text{a.s.}}{\rightarrow}\mu t$ as
$n\rightarrow\infty$ for all $t\geq0$. Because $\bar{B}_{n}(t)$
is nondecreasing in $t$ and $e$ is a continuous function,
the proposition follows from Theorem VI.2.15 in
\citet{JacodShiryaev2002}.

\end{proof}

\begin{lemma}
\label{lemma:L-tilde}
Let
\begin{equation}
\tilde{L}_{n}(t)=\frac{1}{\sqrt{n\gamma_{n}}}\sum_{j=1}^{n}\sum_{k=2}^{N_{j}(\gamma_{n}t)+1}(1-\mu \xi_{j,k})\quad\mbox{for }t\geq0.\label{eq:B_check}
\end{equation}
Under the conditions of Theorem~\ref{theorem:FCLT},
\[
\tilde{L}_{n}\Rightarrow\hat{B}\quad\mbox{as }n\rightarrow\infty\mbox{.}
\]
\end{lemma}

\begin{proof}
Let $\{\eta_{k}:k\in\mathbb{N}\}$ be a sequence of iid
random variables following distribution $F$. Then, $\mu \eta_{k}$ has
mean $1$ and variance $c_{S}^{2}$. Put
\[
\tilde{H}_{n}(t)=\frac{1}{\sqrt{n\gamma_{n}}}\sum_{k=1}^{\lfloor n\gamma_{n}t\rfloor}(1-\mu \eta_{k})\quad\mbox{for }t\geq0.
\]
By Donsker's theorem, $\tilde{H}_{n}\Rightarrow\hat{H}$ as $n\rightarrow\infty$,
where $\hat{H}$ is a driftless Brownian motion with variance $c_{S}^{2}$ and $\hat{H}(0)=0$.
By \eqref{eq:B_bar},
\[
\tilde{H}_{n}(\bar{B}_{n}(t))=\frac{1}{\sqrt{n\gamma_{n}}}\sum_{k=1}^{n\gamma_{n}\bar{B}_{n}(t)}(1-\mu \eta_{k})=\frac{1}{\sqrt{n\gamma_{n}}}\sum_{k=1}^{N_{1}(\gamma_{n}t)+\cdots+N_{n}(\gamma_{n}t)}(1-\mu \eta_{k}).
\]
It follows from Proposition~\ref{proposition:FSLLN} and the random-time-change
theorem that $\tilde{H}_{n}\circ\bar{B}_{n}\Rightarrow\mu^{1/2}\hat{H}$ as $n\rightarrow\infty$.
Because $\tilde{L}_{n}$ has the same distribution as $\tilde{H}_{n}\circ\bar{B}_{n}$
and $\mu^{1/2}\hat{H}$ has the same distribution as $\hat{B}$, the
lemma follows.

\end{proof}

\begin{lemma} \label{lemma:tightness} Under the conditions of Theorem~\ref{theorem:FCLT},
for all $0\leq r\leq s\leq t$ and $n\in\mathbb{N}$, there exists
$0<c<\infty$ such that
\[
\mathbb{E}[(\tilde{B}_{n}(s)-\tilde{B}_{n}(r))^{2}(\tilde{B}_{n}(t)-\tilde{B}_{n}(s))^{2}]\leq c(t-r)^{2}.
\]

\end{lemma}

\begin{proof}
Let $\check{N}_{j}(u)=N_{j}(u)-\mu u$ for $u\geq0$
and $j=1,\ldots,n$. Because $N_{j}$ is a stationary renewal process,
by inequalities (7) and (8) in \citet{Whitt85}, there exists $c_{1}<\infty$
such that
\begin{equation}
\mathbb{E}[(\check{N}_{j}(s)-\check{N}_{j}(r))^{2}]\leq c_{1}(s-r)\label{eq:inequality1}
\end{equation}
and
\begin{equation}
\mathbb{E}[(\check{N}_{j}(s)-\check{N}_{j}(r))^{2}(\check{N}_{j}(t)-\check{N}_{j}(s))^{2}]\leq c_{1}(t-r)^{2}\label{eq:inequality2}
\end{equation}
for all $0\leq r\leq s\leq t$. (The regularity condition (\ref{eq:conditionF})
is required for inequality \eqref{eq:inequality2}.) In
addition, it follows from \eqref{eq:inequality1} and H\"{o}lder's inequality
that
\begin{equation}
\mathbb{E}[|\check{N}_{j}(s)-\check{N}_{j}(r)||\check{N}_{j}(t)-\check{N}_{j}(s)|]\leq c_{1}(s-r)^{1/2}(t-s)^{1/2}\leq c_{1}(t-r).\label{eq:inequality3}
\end{equation}
Because $N_{1},\ldots,N_{n}$ are iid processes,
\begin{align*}
\mathbb{E}[(\tilde{B}_{n}(s)-\tilde{B}_{n}(r))^{2}(\tilde{B}_{n}(t)-\tilde{B}_{n}(s))^{2}]
 & =\frac{1}{n\gamma_{n}^{2}}\mathbb{E}[(\check{N}_{1}(\gamma_{n}s)-\check{N}_{1}(\gamma_{n}r))^{2}(\check{N}_{1}(\gamma_{n}t)-\check{N}_{1}(\gamma_{n}s))^{2}]\\
&\quad+\frac{n-1}{n\gamma_{n}^{2}}\mathbb{E}[(\check{N}_{1}(\gamma_{n}s)-\check{N}_{1}(\gamma_{n}r))^{2}]\mathbb{E}[(\check{N}_{1}(\gamma_{n}t)-\check{N}_{1}(\gamma_{n}s))^{2}]\\
&\quad+\frac{2(n-1)}{n\gamma_{n}^{2}}\mathbb{E}[(\check{N}_{1}(\gamma_{n}s)-\check{N}_{1}(\gamma_{n}r))(\check{N}_{1}(\gamma_{n}t)-\check{N}_{1}(\gamma_{n}s))]^{2}\\
 & \leq c_{1}(t-r)^{2}+c_{1}^{2}(s-r)(t-s)+2c_{1}^{2}(t-r)^{2},
\end{align*}
in which the inequality is obtained by \eqref{eq:inequality1}--\eqref{eq:inequality3}.
The lemma follows with $c=3c_{1}^{2}+c_{1}$.

\end{proof}

\begin{proof}[Proof of Theorem~\ref{theorem:FCLT}.] For $j\in\mathbb{N}$,
let
\begin{equation}
R_{j}(t)=S_{j,N_{j}(t)+1}-t\label{eq:recess}
\end{equation}
be the recess of $N_{j}$ at $t\geq0$. In particular,
\begin{equation}
R_{j}(0)=\xi_{j,1}.\label{eq:recess0}
\end{equation}
Because $N_{1},\ldots,N_{n}$ are iid stationary renewal processes,
$R_{1}(t),\ldots,R_{n}(t)$ are iid random variables following distribution
$F_{e}$ for all $t\geq0$, each having mean
\[
m_{e}=\int_{0}^{\infty}t\,\mathrm{d}F_{e}(t)=\frac{1+c_{S}^{2}}{2\mu}
\]
and variance
\[
\sigma_{e}^{2}=\int_{0}^{\infty}t^{2}\,\mathrm{d}F_{e}(t)-m_{e}^{2}=\frac{\mu}{3}\int_{0}^{\infty}t^{3}\,\mathrm{d}F(t)-m_{e}^{2}.
\]
Note that  $\sigma_{e}^{2}<\infty$ by \eqref{eq:m3}. Let
\[
\tilde{R}_{n}(t)=\frac{1}{\sqrt{n\gamma_{n}}}\sum_{j=1}^{n}(R_{j}(\gamma_{n}t)-m_{e}).
\]
Then,
\[
\mathbb{E}[\tilde{R}_{n}(t)^{2}]=\frac{\sigma_{e}^{2}}{\gamma_{n}}\rightarrow0\quad\mbox{as }n\rightarrow\infty,
\]
which implies that $\tilde{R}_{n}(t)\Rightarrow0$ as $n\rightarrow\infty$
for $t\geq0$. By Theorem 3.9 in \citet{Billingsley99},
\begin{equation}
(\tilde{R}_{n}(t_{1}),\ldots,\tilde{R}_{n}(t_{\ell}))\Rightarrow0\quad\mbox{as }n\rightarrow\infty\label{eq:R_tilde}
\end{equation}
for any $\ell\in\mathbb{N}$ and $0\leq t_{1}<\cdots<t_{\ell}$. By
\eqref{eq:partial-sum}, \eqref{eq:recess}, and \eqref{eq:recess0},
\[
R_{j}(t)=\xi_{j,1}+\sum_{k=2}^{N_{j}(t)+1}\xi_{j,k}-t=R_{j}(0)+\sum_{k=2}^{N_{j}(t)+1}(\xi_{j,k}-\mu^{-1})+\mu^{-1}N_{j}(t)-t.
\]
Then, by \eqref{eq:B_tilde} and \eqref{eq:B_check}, we obtain
\begin{equation}
\tilde{B}_{n}(t)=-\mu\tilde{R}_{n}(0)+\mu\tilde{R}_{n}(t)+\tilde{L}_{n}(t).\label{eq:B_tilde_composition}
\end{equation}
We deduced from \eqref{eq:R_tilde}, \eqref{eq:B_tilde_composition}, and Lemma~\ref{lemma:L-tilde}
that
\[
(\tilde{B}_{n}(t_{1}),\ldots,\tilde{B}_{n}(t_{\ell}))\Rightarrow(\hat{B}(t_{1}),\ldots,\hat{B}(t_{\ell}))\quad\mbox{as }n\rightarrow\infty\mbox{.}
\]
Finally, it follows from Lemma~\ref{lemma:tightness} and Theorem~13.5
in \citet{Billingsley99} (with condition (13.13) replaced by (13.14))
that $\tilde{B}_{n}\Rightarrow\hat{B}$ as $n\rightarrow\infty$.

\end{proof}

\bibliographystyle{ormsv080}
\bibliography{../refs}

\end{document}